\documentclass[a4paper,draft]{amsproc}
\usepackage{amssymb}
\usepackage[hyphens]{url} 

\theoremstyle{plain}
 \newtheorem{thm}{Theorem}[section]
 \newtheorem{prop}{Proposition}[section]
 \newtheorem{lem}{Lemma}[section]
 
\theoremstyle{definition}

\theoremstyle{rem}
 \newtheorem{rem}{Remark}[section]

 \numberwithin{equation}{section}
 
 \allowdisplaybreaks

\renewcommand{\le}{\leqslant}\renewcommand{\ge}{\geqslant}
\renewcommand{\leq}{\leqslant}\renewcommand{\geq}{\geqslant}

\def\N{{\mathbb N}}
\def\R{{\mathbb R}}

\title[Arithmetic functions at factorial arguments]{Arithmetic functions at factorial arguments}



\author{\bfseries Jean-Marie De Koninck} 

\address{ 
D\'epartement de math\'ematiques \\ 
Universit\'e Laval   \\ 
Qu\'ebec\\
Canada}
\email{jmdk@mat.ulaval.ca}

\author{\bfseries William Verreault}
\address{ 
Department of Mathematics \\ 
University of Toronto   \\ 
Toronto\\
Canada}
\email{william.verreault@utoronto.ca}

\thanks{Supported by the Natural Sciences and Engineering Research Council of Canada.} 

\begin{document}

{\begin{flushleft}\baselineskip9pt\scriptsize
MANUSCRIPT
\end{flushleft}}
\vspace{18mm} \setcounter{page}{1} \thispagestyle{empty}

\begin{abstract}
For various arithmetic functions $f:\N \to \R$, the behavior of $f(n!)$ and that of $\sum_{n\le N} f(n!)$ can be intriguing. For instance, for some functions $f$, we have ${f(n!)=\sum_{k\le n}f(k)}$, for others, we have ${f(n!)=\sum_{p\le n}f(p)}$ (where the sum runs over all the primes $p\le n$). Also, for some $f$, their minimum order coincides with $\lim_{n\to \infty}f(n!)$, for others, it is their maximum order that does so.
Here, we elucidate such phenomena and more generally,  we embark on a study of $f(n!)$ and of $\sum_{n\le N}f(n!)$ for a wide variety of arithmetical functions\,$f$. In particular, letting $d(n)$ and $\sigma(n)$ stand respectively for the number of positive divisors of $n$ and the sum of the positive divisors of $n$, we obtain new accurate asymptotic expansions for $d(n!)$ and $\sigma(n!)$. Furthermore, setting $\rho_1(n):=\max\{d\mid n:d\le \sqrt n\}$ and observing that no one has yet obtained an asymptotic value for $\sum_{n\le N} \rho_1(n)$ as $N\to \infty$, we show how one can obtain the asymptotic value of $\sum_{n\le N} \rho_1(n!)$.
\end{abstract}

\maketitle

\section{Introduction} \label{sec:intro}
It is common to enquire about the average order of an arithmetic function $f:\N \to \R$ by first examining the behavior of the sum
$\sum_{n\le N} f(n)$
for large $N$. What if we replace the argument $n$ by $n!$\,? More precisely, given an arithmetic function $f$, how does $f(n!)$ behave for large $n$  and what is the asymptotic value of $\sum_{n\le N} f(n!)$ as $N\to \infty$\,? These are natural questions to ask. Yet, in the literature, there seems to be only a handful of results regarding the study of $f(n!)$ and that of $\sum_{n\le N}f(n!)$ for classical arithmetic functions $f$. One of the first such functions to have been closely examined for its exact order at factorial arguments is the number of positive divisors function $d(n)$. In fact, Ramanujan \cite{ram2, ram1} obtained interesting estimates for $d(n!)$, and later, Erd\H{o}s, Graham, Ivi\'c, and Pomerance \cite{egip} refined his findings, whereas more recently, Jakimczuk \cite{jakim,jakim2} obtained estimates for $\sigma(n!)$, where $\sigma(n)$ stands for the sum  of the positive divisors of $n$.

Here, more generally, we embark on a study of $f(n!)$ and of $\sum_{n\le N} f(n!)$ for a wide variety of classical arithmetic functions $f$. Besides the functions $d(n)$ and $\sigma(n)$ just mentioned, we will pay particular attention to the arithmetic functions
\begin{eqnarray*}
&& \omega(n):=\sum_{p\mid n} p,\quad \Omega(n):=\sum_{p^\alpha\|n} \alpha, \quad \beta(n):=\sum_{p\mid n} p, \quad B(n):=\sum_{p^\alpha\|n}\alpha p,\\
& &  B_1(n):=\sum_{p^\alpha\|n}p^\alpha,\quad \gamma(n):=\prod_{p\mid n} p, \quad \eta(n):=\eta(p_1^{a_1}\cdots p_r^{a_r})=a_1\cdots a_r,
\end{eqnarray*}
the Euler totient function $\phi$ and finally the {\it middle divisors} functions
$$\rho_1(n):=\max\{d\mid n: d\le \sqrt n\}\qquad \mbox{and} \qquad \rho_2(n):=\min\{d\mid n: d\ge \sqrt n\}.$$

In Section \ref{sec:middle}, we examine the case of the middle divisors of $n!$. Then, preliminary results and statements regarding additive and multiplicative functions are given in Sections \ref{sec:additive} and \ref{sec:multiplicative}. In Section \ref{sec:compare}, for various arithmetic functions $f$, we compare the behavior of $f(n!)$ with the maximal order of $f(n)$ or their minimal order. The proofs of the results regarding additive and multiplicative functions are laid out in Section \ref{sec:proofs}.
We conclude with Section \ref{sec:chowla} where we show that the counterpart at factorial arguments of the famous Chowla conjecture holds and we examine the behavior of various similar sums.

As we will see, several of the  estimates that we obtain are almost straightforward whereas others are non trivial  and in fact give way to new results, in particular to new asymptotic expansions for $d(n!)$, $\sigma(n!)$ and $\eta(n!)$.

From here on, the letter $p$ is reserved for primes while the letter $c$ with or without subscript stands for an explicit constant, but not necessarily the same at each occurrence. Also, we let $\pi(x)$ stand for the number of primes not exceeding $x$. Finally, given an arithmetic function $f$, we will often refer to the sum
$$S_f(n):=\sum_{2\le m\le n} f(n).$$

\section{The middle divisors of $n!$} \label{sec:middle}
Before we examine the cases of additive and multiplicative functions, which are somewhat easier to study, we shall first consider a case of arithmetic functions which are neither additive nor multiplicative. We are interested in the middle divisors $\rho_1(n)$ and $\rho_2(n)$ of a positive integer $n$, which are defined as
$$\rho_1(n)  :=  \max\{ d\mid n: d\le \sqrt n\} \qquad \mbox{and} \qquad \rho_2(n)  :=  \min\{ d\mid n: d\ge \sqrt n\}.$$
In particular, $\rho_1(n)\le \sqrt n \le \rho_2(n)$ and $\rho_1(n)\rho_2(n)=n$.

In 1976, Tenenbaum \cite{tenenbaum} proved that
\begin{equation}\label{eq:rho2}
\sum_{n\leq N}\rho_{2}(n)=\dfrac{\pi^{2}}{12}\dfrac{N^{2}}{\log N}\Big(1+O\Big(\dfrac{1}{\log N}\Big)\Big)
\end{equation}
and that for any given $\varepsilon >0$, there exists $N_{0}=N_{0}(\varepsilon)$ such that for all $N\geq N_{0}$,
\begin{equation}\label{eq:rho1}
\dfrac{N^{3/2}}{(\log N)^{\delta+\varepsilon}}<\sum_{n\leq N}\rho_{1}(n)\ll \dfrac{N^{3/2}}{(\log N)^{\delta}(\log \log N)^{1/2}},
\end{equation}
where $\displaystyle{\delta=1-\dfrac{1+\log \log 2}{\log 2}\approx 0.086071}$.

In 2020, De Koninck and Razafindrasoanaivolala \cite{dkr} generalised estimate (\ref{eq:rho2}) by establishing that, given any real number $a>0$ and any positive integer $k$,
\begin{eqnarray*}
\sum_{n\leq N}\rho_{2}(n)^{a}
=c_{0}\dfrac{N^{a+1}}{\log N}+c_{1}\dfrac{N^{a+1}}{\log^{2}N}+\cdots
+c_{k-1}\dfrac{N^{a+1}}{\log^{k}N}+O\Big(\dfrac{N^{a+1}}{\log^{k+1}N}\Big),
\end{eqnarray*}
where, for $\ell=0,1,\ldots,k-1$, the constants $c_{\ell}=c_{\ell}(a)$ are explicitly given.
They also proved that, given any integer $k\geq 1$ and any real number $r> -1$,
$$
\sum_{n\leq N}\dfrac{\rho_{2}(n)}{\rho_{1}(n)^{r}}=h_{0}\dfrac{N^{2}}{\log N}+h_{1}\dfrac{N^{2}}{\log^{2}N}+\cdots
+h_{k-1}\dfrac{N^{2}}{\log^{k}N}+O\left(\dfrac{N^{2}}{\log^{k+1}N}\right)
$$
where $h_{0}=\dfrac{\zeta(r+2)}{2}$ and for each $1\leq \ell\leq k-1$,
$$
h_{\ell}=\left(\dfrac{r+2}{2}\right)c_{\ell}+\sum_{\nu=0}^{\ell-1}\dfrac{rc_{\nu}}{2}\prod_{m=\nu}^{\ell-1}\left(\dfrac{m+1}{2}\right),$$
with, for each $\nu=0,1,\ldots,\ell$,
$$c_{\nu}=\dfrac{\nu!}{(r+2)^{\nu+1}}\sum_{j=0}^{\nu}\dfrac{(r+2)^{j}(-1)^{j}\zeta^{(j)}(r+2)}{j!}.$$

In 2008, Ford \cite{ford} considerably improved (\ref{eq:rho1}) by showing that
$$
\sum_{n\leq N}\rho_{1}(n)\asymp \dfrac{N^{3/2}}{(\log N)^{\delta}(\log \log N)^{3/2}}.
$$
Interestingly, no asymptotic formula for the sum $\sum_{n\leq N}\rho_{1}(n)$ has yet been found.

Recently, De Koninck and Razafindrasoanaivolala \cite{jm-arthur-2023} established that
$$
\sum_{\substack{4\le  n\le N \\ n\neq \text{prime}}} \frac{\log \rho_2(n)}{\log \rho_1(n)} = N \log \log N + O(N)
$$
and also that, for all $N$ sufficiently large,
$$
c_1\, N < \sum_{2\le n\le N } \frac{\log \rho_1(n)}{\log \rho_2(n)} < c_2\, N,
$$
where
\begin{eqnarray*}
c_1 & = &  1-\log 2 + \int_1^2 \frac{1-\log u}{u(u+1)}\,du + \int_3^\infty \frac{u-1}{u+1} \frac{\rho(u-1)}u\,du \approx 0.528087, \\
c_2 & = & 2-2\log 2 \approx 0.613706.
\end{eqnarray*}

As we will now see, the behavior of each of the six sums
\begin{equation} \label{eq:liste}
\sum_{n\le N} \rho_1(n!), \sum_{n\le N} \rho_2(n!),  \sum_{n\leq N}\dfrac{\rho_{1}(n!)}{\rho_{2}(n!)},  \sum_{n\leq N}\dfrac{\rho_{2}(n!)}{\rho_{1}(n!)}, \sum_{2\le n\le N} \frac{\log \rho_1(n!)}{\log \rho_2(n!)},   \sum_{2\le n\le N} \frac{\log \rho_2(n!)}{\log \rho_1(n!)}
\end{equation}
is more manageable.

\begin{thm} \label{thm:middle-1}
Letting $c$ be the positive constant appearing below estimate (\ref{eq:tij}),  we have
$$
\sum_{n\le N} \rho_1(n!) =  \sqrt{N!}\Big(1 + O \Big(\frac 1{N^c}\Big)\Big) \quad \mbox{ and }\quad
\sum_{n\le N} \rho_2(n!)   =  \sqrt{N!}\Big(1 + O \Big(\frac 1{N^c}\Big)\Big).
$$
\end{thm}

\begin{thm} \label{thm:middle-2}
Letting $c$ be as above, we have
$$
\sum_{n\le N} \frac{\rho_1(n!)}{\rho_2(n!)}  =  N + O \left( N^{1-c} \right) \quad\mbox{ and }\quad
\sum_{n\le N} \frac{\rho_2(n!)}{\rho_1(n!)}  =  N + O \left( N^{1-c} \right).
$$
\end{thm}

\begin{thm} \label{thm:middle-3}
We have
$$
\sum_{2\le n\le N} \frac{\log \rho_1(n!)}{\log \rho_2(n!)} =  N + O(1)\quad\mbox{ and }\quad
\sum_{2\le n\le N} \frac{\log \rho_2(n!)}{\log \rho_1(n!)}  =  N + O(1) .
$$
\end{thm}

Observe that the main difficulty in finding an estimate for any of the sums appearing  in (\ref{eq:liste}) at arguments $n$ (and not $n!$) is that the values of the function $\rho_1(n)$ (resp. $\rho_2(n)$) are spread out over the interval $[1,\sqrt n]$ (resp. $[\sqrt n,n]$). On the contrary, the values of $\rho_1(n!)$ and  $\rho_2(n!)$ are all near $\sqrt{n!}$, as is made clear in the following proposition.

\begin{prop}\label{prop:d1d2}
To any given sufficiently large integer $n$, one can associate two positive integers $d_1<d_2$ both dividing $n!$ and satisfying
\begin{equation} \label{eq:close-1}
\Big(1- \frac 1{n^c} \Big) \sqrt{n!} < d_1 < \sqrt{n!} < d_2 < \Big(1+ \frac 2{n^c} \Big)\sqrt{n!},
\end{equation}
where $c$ is the constant appearing below estimate (\ref{eq:tij}),
thereby implying that
\begin{equation} \label{eq:close-2}
\Big(1- \frac 1{n^c}\Big) \sqrt{n!} < \rho_1(n!) < \sqrt{n!} < \rho_2(n!) <\Big(1+ \frac 2{n^c} \Big) \sqrt{n!}.
\end{equation}
\end{prop}

\subsection*{Numerical data}
In the proof of Proposition \ref{prop:d1d2} stated in Section \ref{sec:mid1}, we will create the numbers $d_1=d_1(n)$ alluded to in \eqref{eq:close-1} so that $d_1\le \rho_1(n)$. Interestingly, one can check that $d_1=\rho_1(n)$ in the cases $n=7,10,22$. These are perhaps the only ones with that particular property. In the particular cases $n=10,20,30,40$,  we obtain the following data.
\vskip 5pt

\noindent
\begin{footnotesize}
\noindent
\begin{tabular}{|c|c|c|c|c|}\hline
$n$ & $n!$         & $d_1$ & $\lfloor \sqrt{n!} \rfloor$ & $\lfloor \sqrt{n!} \rfloor/d_1$ \\ \hline  \hline
10  & 3\,628\,800                 & 1890  & 1904 & 1.007407407    \\
    & $2^8\cdot 3^4\cdot 5^2\cdot 7$ & $2\cdot 3^3 \cdot 5\cdot 7$ & &  \\ \hline
20 & 2\,432\,902\,008\,176\,640\,000 & 1\,558\,878\,750 &  1\,559\,776\,268 & 1.000575746    \\
   & $2^{18}\cdot 3^8\cdot 5^4\cdot 7^2\cdot 11\cdot 13\cdot 17\cdot 19$ & $2\cdot 3^3\cdot 5^4\cdot 11\cdot 13\cdot 17\cdot 19$ & &  \\ \hline
\end{tabular}
\end{footnotesize}
\vskip 5pt

\noindent
When $n=30$,
\begin{eqnarray*}
n! & = & 265\,252\,859\,812\,191\,058\,636\,308\,480\,000\,000\\
 & & = 2^{26}\cdot 3^{14}\cdot 5^7\cdot 7^4\cdot 11^2\cdot 13^2\cdot 17\cdot 19\cdot 23\cdot 29,\\
d_1 & = & 16\,283\,616\,779\,520\,000 = 2^{11}\cdot 3^{10}\cdot 5^4\cdot 17\cdot 19\cdot 23\cdot 29, \\
\lfloor \sqrt{n!} \rfloor & = &  16\,286\,585\,271\,694\,955, \\
\lfloor \sqrt{n!} \rfloor/d_1 & = & 1.000182299.
\end{eqnarray*}
\goodbreak 
\vskip 5pt

\noindent
When $n=40$,
\begin{eqnarray*}
n! & = & 815\,915\,283\,247\,897\,734\,345\,611\,269\,596\,115\,894\,272\,000\,000\,000\\
 & & = 2^{38} \cdot 3^{18} \cdot 5^9 \cdot 7^5 \cdot 11^3 \cdot 13^3 \cdot 17^2 \cdot 19^2 \cdot 23 \cdot 29 \cdot 31 \cdot 37,\\
d_1 & = & 903\,280\,055\,273\,332\,645\,708\,800 \\
& & = 2^{10}\cdot 3^{15}\cdot 5^2\cdot 7\cdot 11\cdot 13^3\cdot 19\cdot 23\cdot 29\cdot 31\cdot 37, \\
\lfloor \sqrt{n!} \rfloor & = &  903\,280\,290\,523\,322\,408\,635\,610, \\
\lfloor \sqrt{n!} \rfloor/d_1 & = & 1.000000260.
\end{eqnarray*}
Observe that $\rho_1(10!)=1\,890$, $\rho_1(20!)=1\,559\,376\,000$,\\
$\rho_1(30!)=16\,286\,248\,192\,500\,000$, and $\rho_1(40!)=
903\,280\,055\,273\,332\,645\,708\,800$.

\section{The case of additive functions} \label{sec:additive}
As a first example of additive functions, let us consider the function $f(n)=\log n$, for which we have $f(n!)=\sum_{k\leq n}\log k=S_{\log}(n)$.  As we will see, for large families of functions $f$, we do have that $f(n!)=S_f(n)$.

\subsection{Small additive functions}
We first consider the case of the ``small'' additive functions  $\omega(n)$ and $\Omega(n)$.

\begin{thm} \label{thm:add1} We have
$$\omega(n!)=\pi(n) \qquad \text{and}\qquad \Omega(n!)=S_{\Omega}(n).$$
Moreover,
\begin{eqnarray}\label{eq:th21a}
\omega(n!) & = & \frac{n}{\log n}\left(1+ O \left(\frac 1{\log n} \right) \right),\\ \label{eq:th21b}
\Omega(n!) & = & n\left(\log\log n +c_1+c_2+ O \left( \frac 1{\log n} \right) \right),
\end{eqnarray}
where
$c_1 = \gamma + \sum_p \big( \log(1- \frac 1p) + \frac 1p \big)$ and $c_2 = \sum_p \frac 1{p(p-1)}$ (here,  $\gamma$ stands for the Euler–Mascheroni constant).
\end{thm}

The proof of Theorem \ref{thm:add1} will be given in Section  \ref{sec:proofs}. Beforehand, we should mention that the average order of each of the functions $\omega(n)$ and $\Omega(n)$ has been studied with great accuracy. In particular, it is well known  (see for instance formulas (6.11) and (6.12) in
the book of De Koninck and Luca \cite{DL}) that
\begin{eqnarray}
\sum_{n\le N} \omega(n) & = & N\Big(\log \log N + c_1 + O \Big( \frac 1{\log N} \Big) \Big), \nonumber \\
\label{eq:Omega}
\sum_{n\le N} \Omega(n) & = & N\Big(\log \log N +  c_1 + c_2 + O \Big( \frac 1{\log N} \Big) \Big).
\end{eqnarray}
Similarly, it was proved by De Koninck \cite{JMDK-1972} that
\begin{eqnarray*}
\sum_{2\le n\le N} \frac 1{\omega(n)} & = & \frac N{\log \log N} \Big(1 + O \Big( \frac 1{\log \log N} \Big) \Big),\\
\sum_{2\le n\le N} \frac 1{\Omega(n)} & = & \frac N{\log \log N} \Big(1 + O \Big( \frac 1{\log \log N} \Big) \Big).
\end{eqnarray*}
In fact, more accurate formulas can be found in Chapter 5 of the book of De Koninck and Ivi\'c \cite{book-JM-AI}.
In addition, De Koninck \cite{JMDK-1974} proved that
$$\sum_{2\le n \le N} \frac{\Omega(n)}{\omega(n)} = N + c_2 \frac N{\log \log N} + O \Big( \frac N{(\log \log N)^2} \Big).$$

On the other hand, as we will see,  the behavior of the sum $\sum_{2\leq n\leq N}1/\pi(n)$ turns out to have a connection with the function $\omega(n!)$, and interestingly the  study of that particular sum has a long history, a survey of which can be found in \cite{JMDK-2021}. The best known estimate is due to Ivi\'c \cite{Ipi} and can be stated as follows:
\begin{equation} \label{eq:ivicprime}
\sum_{2\leq n\leq N}\frac{1}{\pi(n)} = \frac 12 \log^2N - \log N - \log \log N +c_3 + O \Big( \frac 1{\log N} \Big),
\end{equation}
where $c_3$ is an absolute constant that is known  to satisfy $6.6840< c_3 <6.7830$ (see Berkane and Dusart \cite{BD}).

Analogous results for  sums of the ``small'' additive functions $\omega(n)$ and $\Omega(n)$  at factorial arguments are fairly easy to obtain. In fact, as we will see in Section \ref{sec:proofs}, one can obtain most of the following results by combining the above estimates with Theorem \ref{thm:add1}.
\begin{thm} \label{thm:3}
 We have
\begin{eqnarray}
\sum_{n\leq N} \omega(n!) & = & \frac{N^2}{2\log N}\Big(1+O\Big(\frac{1}{\log N}\Big)\Big), \label{eq:as1} \\
\sum_{ n\le N} \Omega(n!) & = & \frac{N^2}{2} \Big(\log \log N + c_1+c_2 + O\Big(\frac{1}{\log N}\Big)\Big), \label{eq:as2}\\
\sum_{2\le n\le N} \frac 1{\omega(n!)} & = & \frac 12 \log^2N - \log N - \log \log N +c_3 + O \Big( \frac 1{\log N} \Big), \label{eq:as3}\\
\sum_{2\le n\le N} \frac 1{\Omega(n!)} & = & \frac{\log N}{\log \log N}\Big(1+O\Big(\frac{1}{\log\log N}\Big)\Big), \label{eq:as4}\\
\sum_{2\le n \le N} \frac{\Omega(n!)}{\omega(n!)} & = & N(\log N \log \log N + (c_1+c_2)\log N + O(\log\log N)). \label{eq:as5}
\end{eqnarray}
\end{thm}

\subsection{Large additive functions}
Next, we consider the ``large'' additive functions
$\beta(n)$, $B(n)$ and $B_1(n)$ defined in Section 1.
For these, we have the following.

\begin{thm} \label{thm:add2} For each integer $n\ge 2$,
\begin{equation} \label{eq:th23}
\beta(n!)=\sum_{p\le n} p,\qquad B(n!)=S_B(n),\qquad \text{and}\qquad B_1(n!)=\sum_{p\le n} p^{\alpha_p(n)},
\end{equation}
where $\alpha_p(n)$ is defined below in \eqref{eq:polignac}.
Moreover, for all $n\ge 2$,
\begin{eqnarray} \label{eq:th23a}
\beta(n!) & = & \frac{n^2}{2\log n}\left(1 + O \left( \frac 1{\log n} \right) \right),\\ \label{eq:th23b}
B(n!) & = &  \frac{\pi^2}{12}\frac{n^2}{\log n}\left(1 + O \left( \frac 1{\log n} \right) \right),\\ \label{eq:th23c}
B_1(n!) & = &   2^{n-s_2(n)} +O\left((\sqrt 3)^n\right),
\end{eqnarray}
where $s_2(n)$ stands for the sum of the binary digits of $n$.
\end{thm}

Let us recall that
Alladi and Erd\H{o}s \cite{alladi-erdos} have proved that
\begin{equation} \label{eq:alladi-erdos}
\sum_{n\le N} \beta(n) \sim \frac{\pi^2}{12} \frac{N^2}{\log N}  \qquad (N\to \infty),
\end{equation}
and that De Koninck and Ivi\'c  \cite{jm-a} later showed that \eqref{eq:alladi-erdos} can be improved to
$$\sum_{n\le N} \beta(n)   = \sum_{i=1}^M  \frac{b_i N^2}{\log^i N} + O\left( \frac{N^2}{\log^{M+1} N} \right),$$
where $M$ is any preassigned positive integer and where each $b_i$ is a computable constant, with $b_1=\pi^2/12$.
Regarding the function $B$ and $B_1$,  Ivi\'c \cite{ivic-2003} proved that
$$\sum_{n\le N} B(n) = \sum_{n\le N} \beta(n) +O(N^{2/3})\quad
\mbox{ and } \quad \sum_{n\le N} B_1(n) = \sum_{n\le N} \beta(n) +O(N^{2/3}),$$
from which it follows in particular that
\begin{eqnarray} \label{eq:BB1}
\sum_{n\le N} B(n) & = & \frac{\pi^2}{12} \frac{N^2}{\log N}  \left( 1+ O \left( \frac 1{\log N} \right)\right) \\ \label{eq:BB2}
\sum_{n\le N} B_1(n) & = &\frac{\pi^2}{12} \frac{N^2}{\log N} \left( 1+ O \left( \frac 1{\log N} \right)\right).
\end{eqnarray}
Incidently, Ivi\'c \cite{ivic-1980} proved that
$$\sum_{2\le n \le N} \frac 1{\beta(n)} = \frac N{e^{\sqrt{2\log N \log \log N} +O(\sqrt{\log N \log \log \log N})}},$$
an estimate that was later slightly improved by Ivi\'c and Pomerance \cite{ivic-pomerance}.

The factorial counterparts of \eqref{eq:alladi-erdos}, \eqref{eq:BB1} and \eqref{eq:BB2} are as follows.
\begin{thm} \label{thm:2}
As $N$ becomes large,
\begin{eqnarray}
\sum_{2\le n \le N} \beta(n!) &=& \frac{N^3}{6\log N}\left( 1+ O \left( \frac 1{\log N} \right)\right), \label{eq:al1} \\
\sum_{2\le n \le N} B(n!) &=& \frac{\pi^2}{36}\frac{N^3}{\log N} \left( 1+ O \left( \frac 1{\log N} \right)\right), \label{eq:al2} \\
\sum_{2\le n \le N} B_1(n!) &=& 2^{N+O(\log N)} .\label{eq:al3}
\end{eqnarray}
Furthermore, letting $f$ be any one of the three functions $\beta$, $B$ and $B_1$, we have
$$\sum_{n=2}^\infty  \frac 1{f(n!)} <\infty.$$
\end{thm}

\section{The case of multiplicative functions} \label{sec:multiplicative}
\subsection{Classical sums}
Let us first consider the multiplicative functions $\gamma$ and $\phi$ already mentioned in Section 1.

\begin{thm} \label{thm:mult1}
We have
\begin{equation} \label{eq:multfuncest}
\gamma(n!)=\prod_{p\leq n}p\qquad \text{and}\qquad \phi(n!)=n!\prod_{p\leq n}(1-1/p).
\end{equation}
Moreover,
\begin{equation} \label{eq:multfuncas}
\gamma(n!)=e^{(1+o(1))n}\qquad  \text{and} \qquad \phi(n!)\sim n!\frac{e^{-\gamma}}{\log n} \qquad (n\to\infty).
\end{equation}
\end{thm}

Perhaps the most
challenging estimates of $f(n!)$ come from multiplicative functions $f$ that are neither completely nor strongly multiplicative,
for reasons that will be detailed in Section \ref{sec:proofs1}.
Of particular interest are the functions $d(n)$ and $\sigma(n)$, as well as their higher-order variants $d_k(n)$ and $\sigma_\kappa(n)$ for integers $k\geq 2$ and real numbers $\kappa$. Here $d_k(n):=\sum_{a_1\cdots a_k=n}1$ is the number of ways of representing $n$ as a product of $k$ positive integers and $\sigma_\kappa(n):=\sum_{d\mid n}d^\kappa$ is the sum of the $\kappa$-th powers of the positive divisors of $n$.

Ramanujan \cite{ram2} studied the behavior of $d(n!)$ by first conjecturing that for any given $\varepsilon>0$,
$$
C^{\frac{n}{\log n}(1-\varepsilon)}<d(n!)<C^{\frac{n}{\log n}(1+\varepsilon)},
$$
where $C=(1+1)\sqrt{1+\frac{1}{2}}\sqrt[3]{1+\frac{1}{3}}\sqrt[4]{1+\frac{1}{4}}\cdots \approx 3.51750$,
and by later proving \cite{ram1} the more explicit formula
\begin{equation} \label{eq:dn!4}
d(n!)=C^{\frac{n}{\log n}+o\left(\frac{n}{\log^2n}\right)} \qquad (n\to\infty).
\end{equation}
See Andrews and Berndt \cite{andrews} for more on these estimates.
Much later, Erd\H{o}s, Graham, Ivi\'{c}, and Pomerance \cite{egip} obtained an asymptotic series expansion for $\log d(n!)$ by showing that for any given integer $M\geq 0$,
\begin{equation*}
d(n!)=\exp\left\{\frac{n}{\log n}\sum_{\ell=0}^M\frac{c_\ell}{\log^\ell(n)} + O\Big(\frac{n}{\log^{M+2}n}\Big)\right\},
\end{equation*}
where $\displaystyle{c_\ell=\int_1^\infty\frac{\log(\left\lfloor t\right\rfloor+1)}{t^2}\log^\ell t\,dt}$. In particular,
\begin{equation} \label{eq:5}
\log d(n!)\sim c_0\frac{n}{\log n} \qquad (n\to\infty),
\end{equation}
where $c_0$ can be shown to be equal to $\log C$. This is because, on the one hand,
\begin{eqnarray*}
c_0 & = &\int_1^\infty \frac{\log(\left\lfloor t\right\rfloor+1)}{t^2}\,dt=\sum_{s=1}^\infty\int_{s}^{s+1}\frac{\log(s+1)}{t^2}\,dt \\
& = & \sum_{s=2}^\infty \log s \int_{s-1}^s\frac{dt}{t^2}=\sum_{s=2}^\infty \frac{\log s}{s(s-1)},
\end{eqnarray*}
while on the other hand,
\begin{eqnarray*}
\log C & = & \sum_{s=1}^\infty \frac{1}{s}\log\Big(1+\frac{1}{s}\Big)=\sum_{s=1}^\infty \frac{1}{s}(\log(s+1)-\log s) \\
& = & \sum_{s=2}^\infty \Big(\frac{1}{s-1}-\frac{1}{s}\Big)\log s = \sum_{s=2}^\infty \frac{\log s}{s(s-1)}.
\end{eqnarray*}
In fact, as we will see, each $c_\ell$ is given by a similar explicit formula, and furthermore
\begin{equation} \label{eq:factclapp}
c_\ell=(\ell+1)!+O\left((1-2^{-\ell-1})\ell!\right).
\end{equation}
We shall generalize these results to $d_k(n!)$.
\begin{thm} \label{thm:dk}
Given integers $k\geq 2$ and $M\geq 0$, we have
    $$
d_k(n!) = \exp\left\{\frac{n}{\log n}\sum_{\ell=0}^M\frac{c^{(k)}_\ell}{\log^\ell n} + O\Big(\frac{n}{\log^{M+2}n}\Big)\right\},
    $$
where
\begin{equation} \label{eq:clk}
c_{\ell}^{(k)} = \ell!\sum_{j=0}^\ell\frac{1}{j!}\sum_{s=1}^\infty \frac{\log^j s}{s}\log\Big(1+\frac{k-1}{s}\Big).
\end{equation}
Moreover,
$$
c_\ell^{(k)}=(k-1)(\ell+1)! + O\big((k-1)^2(1-2^{-\ell-1})\ell!\big).
$$
\end{thm}

In particular, Theorem \ref{thm:dk} implies that for each integer $k\ge 2$,
\begin{equation} \label{eq:dkas}
\log d_k(n!)\sim c_0^{(k)}\frac{n}{\log n} \qquad (n\to\infty).
\end{equation}
\begin{rem}
  Equation \eqref{eq:clk} allows one to estimate each of the constants $c_\ell^{(k)}$ for fixed $\ell$ and $k$. For instance, the first five values in the case $k=2$ are ${c_0^{(2)}\approx 1.2578}$, $c_1^{(2)}\approx 2.1139$, $c_2^{(2)}\approx 6.1145$, $c_3^{(2)}\approx 24.1764$, $c_4^{(2)}\approx 120.3601$. It is readily seen that approximation \eqref{eq:factclapp} is rapidly accurate.
  \end{rem}

In the same vein,
Fedorov \cite{FE} studied the number of divisors of the central binomial coefficient $ \binom{2n}{n}$ and obtained a similar asymptotic series. These results are easily extended to the function $d_k(n)$.

\begin{thm} \label{thm:central}
    Given integers $k\geq 2$ and $M\geq 0$, for each integer $n\ge 2$, we have
    \begin{equation} \label{eq:thmcen}
\frac{1}{\log k} \log d_k\Big(\binom{2n}{n}\Big) = \pi(2n)-\pi(n) + \frac{n}{\log n}\sum_{\ell=0}^M\frac{b_\ell}{\log^\ell n} + O\Big(\frac{n}{\log^{M+2}n}\Big),
    \end{equation}
    where
$\displaystyle{b_\ell = \sum_{n=1}^\infty \int_{n+1/2}^{n+1}\frac{\log^{\ell}t}{t^2}\,dt}$.
\end{thm}
\begin{rem}
Note that the right-hand side of \eqref{eq:thmcen} does not depend on $k$. Moreover, we can also deduce from this equality and the computation $b_0=\log 4-1$ that
$$
\log d_k\Big(\binom{2n}{n}\Big) \sim \log k \cdot \log 4\cdot \frac{n}{\log n}\qquad (n\to\infty).
$$
\end{rem}

It is clear that the same ideas can be used with the multiplicative function $\eta$ already defined in Section \ref{sec:intro} and for which we have  $\displaystyle{\log \eta(n!)\sim e_0\frac{n}{\log n}}$, where
$\displaystyle{e_0=\int_1^\infty \frac{\log \lfloor t\rfloor}{t^2}\,dt = \sum_{s=1}^\infty \frac{\log s}{s(s+1)}}$.
This asymptotic formula was first obtained by Jakimczuk \cite{jakim3}, but it was stated in a slightly different form, much weaker than the full statement we will now provide.
\begin{thm} \label{thm:rho}
Given any integer  $M\geq 0$, we have
    $$
\eta(n!) = \exp\left\{\frac{n}{\log n}\sum_{\ell=0}^M\frac{e_\ell}{\log^\ell n} + O\left(\frac{n}{\log^{M+2}n}\right)\right\},
    $$
where $\displaystyle{e_{\ell}= \ell!\sum_{j=0}^\ell\frac{1}{j!}\sum_{s=1}^\infty \frac{\log^j (s+1)}{s+1}\log\Big(1+\frac{1}{s}\Big)}$.
\end{thm}

As for the sum of the positive divisors of $n!$, it was studied by
Jakimczuk \cite{jakim, jakim2}, who obtained that
\begin{align} \label{eq:j}
\sigma(n!)\sim e^{\gamma} n!\log n \qquad (n\to\infty),
\end{align}
and, more explicitly, that
$$
\sigma(n!)= e^{\gamma}\sqrt{2\pi n}\Big(\frac{n}{e}\Big)^n\log n\Big(1+O\Big(\frac{1}{\log n}\Big)\Big).
$$
The proof of these results relies on convoluted arguments. Here, we simplify the proof and consider the more general function $\sigma_\kappa(n!)$ for any fixed $\kappa>1$, while at the same time improving the error term.

\begin{thm} \label{thm:sigmak}
    For all integers $n\ge 2$,
    $$
\sigma(n!)=n!\, e^\gamma\,\log n \Big( 1 + O\Big( \frac 1{\log^3 n} \Big)\Big),
    $$
and for any real $\kappa> 1$,
\begin{equation*} \label{eq:sigmak}
\sigma_\kappa(n!)=(n!)^\kappa\,\zeta(\kappa)\Big( 1 + O\Big( \frac {\log n}{n} \Big)\Big).
\end{equation*}
\end{thm}
Here $\zeta(\cdot)$ is the Riemann zeta function
\begin{equation} \label{eq:zeta}
\zeta(s) = \sum_{n\geq 1}\frac{1}{n^s} = \prod_p\left(1-\frac{1}{p^s}\right)^{-1}\qquad (\text{Re}(s)>1).
\end{equation}

\subsection{Sums of multiplicative functions running over factorials}
\label{subsec:mult-fac}
According to Wintner's theorem (see Theorem 6.13 in the book of De Koninck and Luca \cite{DL}), if two arithmetic functions $f(n)$ and $g(n)$ are connected through the relation
$$\sum_{n=1}^\infty \frac{f(n)}{n^s} = \zeta(s) \sum_{n=1}^\infty \frac{g(n)}{n^s} \quad \mbox{for all real $s>1$}$$
and if moreover $\sum_{n=1}^\infty g(n)/n$ converges absolutely, then $\sum_{n\le N} f(n) =(c+o(1))N$ as $N\to \infty$, where
$c:=\sum_{n=1}^\infty g(n)/n$. Applying this theorem to the multiplicative function $\gamma(n)/n$, one finds that
\begin{equation} \label{eq:gamma-10}
\sum_{n\le N} \frac{\gamma(n)}n = (c+o(1))N, \quad \mbox{where }  c=\prod_p\left(1-\frac{1}{p(p+1)}\right)
\end{equation}
(see \cite[Problem 6.5]{DL} for the details). Using (\ref{eq:gamma-10}) and partial summation, we obtain that
$$
\sum_{n\leq N} \gamma(n) = \left(\frac c2 +o(1)\right)N^2\qquad (N\to\infty).
$$
Also, it is known (see Theorems 320, 324, and 330 in
Hardy and Wright \cite{HW}) that
\begin{eqnarray*}
\sum_{n\le N} \phi(n) & =  & \frac 3{\pi^2} N^2 + O(N\log N),\\
\sum_{n\leq N} d(n) & =  & N\left(\log N + 2\gamma-1+O(N^{-1/2})\right),\\
 \sum_{n\le N} \sigma(n) & =  & \frac{\pi^2}{12} N^2 + O(N\log N).
 \end{eqnarray*}
Moreover, one can prove (see Problem 6.12 and Theorem 6.19 in \cite{DL}) that
$$
\sum_{n\leq N}\log d(n)=\log 2\cdot N\Big(\log\log N + c_1 + c_4 + O\Big(\frac{1}{\log N}\Big)\Big),
$$
where $c_1=\gamma + \sum_p \big( \log(1- \frac 1p) + \frac 1p \big)$ and $c_4=\sum_p \frac{\log(1+1/p)}{p(p-1)}$.
For the sums of $\sigma(n)/\phi(n)$ and $\phi(n)/\sigma(n)$, we have  the following.
\begin{prop} \label{prop:sigphi}
    Given any arbitrarily small number $\delta>0$,
\begin{equation}\label{eq:1}
\sum_{n\le N} \frac{\sigma(n)}{\phi(n)} = d_1 N + O(N^{\frac 12 + \delta}) \qquad \mbox{and} \qquad
\sum_{n\le N} \frac{\phi(n)}{\sigma(n)} = d_2 N + O(N^{\frac 12 + \delta}),
\end{equation}
where
\begin{eqnarray*}
d_1 & = & \prod_p \Big(1 + \frac{2p^2-1}{p(p+1)(p-1)^2} \Big) \approx 3.61744, \\
d_2 & = & \prod_p \Big( 1 - \frac 2{p(p+1)} - \frac{p-1}{p^2} \Big( \frac 1{(p+1)(p^2+p+1)} \\
& & \qquad \qquad + \frac 1{(p^2+p+1)(p^3+p^2+p+1)} + \cdots \Big) \Big)\\
& \approx & 0.45783.
\end{eqnarray*}
\end{prop}

Although we do not claim that the results in Proposition \ref{prop:sigphi} are new, we could not find them in the literature, so for the sake of completeness we shall prove them in Section \ref{sec:pfsum}.

For sums running over factorials, we have the following.

\begin{thm} \label{thm:6} We have
\begin{eqnarray}
\sum_{n\le N} \gamma(n!) & = & e^{N + O \big( \frac N{\log N} \big) }, \label{eq:m1} \\
\sum_{n\le N} \phi(n!) & = & e^{-\gamma}\frac{N!}{\log N}\Big(1+O\Big(\frac{1}{\log^3 N}\Big)\Big), \label{eq:m2}\\
\sum_{ n\le N} \log d(n!) & = & \frac{c_0}2\frac{N^2}{\log N}\Big(1+O\Big(\frac{1}{\log N}\Big)\Big), \label{eq:m3}\\
\sum_{n\le N} \sigma(n!) & = & e^{\gamma}N!\log N\Big(1+O\Big(\frac{1}{\log^3 N}\Big)\Big), \label{eq:m4}\\
\sum_{n\le N} \frac{\sigma(n!)}{\phi(n!)} & = & e^{2\gamma} \left( N\log^2N -2N \log N +2N \right) \Big(1 + O \Big( \frac 1{\log^3 N} \Big) \Big), \label{eq:m5}\\
\sum_{n\le N} \frac{\phi(n!)}{\sigma(n!)} & = & e^{-2\gamma} \frac N{\log^2 N} \Big(1 + O \Big( \frac 1{\log^3 N} \Big) \Big). \label{eq:m6}
\end{eqnarray}
\end{thm}

\section{Comparing $f(n!)$ with the extremal orders of $f(n)$} \label{sec:compare}
In Sections \ref{sec:additive} and \ref{sec:multiplicative}, we compared estimates of sums of additive and multiplicative functions with their factorial counterparts. It is also interesting to compare the behavior of $f(n!)$ with the maximum and minimum values  of $f(n)$.

In this section, we will use the estimate
\begin{equation} \label{eq:Stirling}
\log n!=n\log n - n +O(\log n),
\end{equation}
which follows from the fact that $\sum_{k\leq n}\log k=\int_1^n\log t\,dt+O(\log n)$, which itself represents a weak form of Stirling's formula
\begin{equation} \label{eq:stirling}
n!\sim \Big(\frac{n}{e}\Big)^n \sqrt{2\pi n} \qquad (n\to\infty).
\end{equation}
Estimate \eqref{eq:Stirling} thus implies that, as $n\to \infty$,
\begin{equation}\label{eq:stirling2}
\log n\sim \log\log n! \qquad  \text{and} \qquad  n\sim \frac{\log n!}{\log\log n!}.
\end{equation}

Sometimes $f(n!)$ is very far from the maximum of $f(n)$ over $n\leq N$. For instance, $\max_{n\leq N}\Omega(n)=\lfloor \log N/\log 2\rfloor,$ with equality if and only if $k$ is the largest positive integer such that $n=2^k\leq N$. However, by combining Theorem \ref{thm:add1} and \eqref{eq:stirling2}, we get that
$$
\Omega(n!)\sim \log n!\frac{\log\log\log n!}{\log\log n!} \qquad (n\to\infty),
$$
so that for large $n$, $\Omega(n!)<\log(n!)$, whereas for infinitely many integers $n$ we have $\Omega(n)>1.4 \log n$.

In other instances,  $f(n!)$ is almost as large as $f(n)$. Indeed,
we know that integers $n\leq N$ have at most $(1+o(1))\frac{\log N}{\log\log N}$ distinct prime factors, with equality if and only if $n$ is the product of the first $y$ primes, where $y$ is chosen maximally so that $n=\prod_{p\leq y}p\leq N$. We thus have $\omega(n)=\pi(y)$, and since $\omega(n!)=\pi(n)$, we see that it is almost of this form, and indeed we have
$$
\omega(n!)\sim \frac{\log n!}{(\log\log n!)^2}\qquad (n\to\infty).
$$
There is a similar pattern for $\gamma(n)$. Indeed, it is clear that $\max_{n\leq x}\gamma(n)=\prod_{p\leq y}p$ where $y$ is chosen as above. Thus we see that $\max_{n\leq x}\gamma(n)=\gamma(y!)=x^{1+o(1)}$ as $x\to\infty$, since $y\sim \log x$. On the other hand, $$
\gamma(n!)\sim(n!)^{\frac{1+o(1)}{\log\log n!}} \qquad (n\to\infty).
$$
Similarly, for $d(n)$, a classical result of Wigert \cite{wig} gives
$$
\log d(n)\leq \frac{\log 2\log n}{\log\log n} + O\Big(\frac{\log n}{(\log\log n)^2}\Big),
$$
with equality if and only if $n$ is the product of the first $r$ primes with $r$ large, while we see from \eqref{eq:5} and \eqref{eq:stirling2} that
$$
\log d(n!)\sim c_0\frac{\log n!}{(\log\log n!)^2} \qquad (n\to\infty),
$$
as already observed in \cite{egip}. More generally,
$$
\log d_k(n)\leq (\log k+o(1))\frac{\log n}{\log\log n} \qquad (n\to\infty),
$$
and we get from \eqref{eq:dkas} that
$$
\log d_k(n!)\sim c^{(k)}_0\frac{\log n!}{(\log\log n!)^2} \qquad (n\to\infty).
$$

On the other hand,  sometimes  $f(n!)$ attains the minimal or maximal value of $f(n)$. For example, we have
$$
\phi(n!)\sim e^{-\gamma}\frac{n!}{\log\log n!} \qquad \text{and} \qquad \sigma(n!)\sim e^\gamma n!\log\log n! \qquad (n\to\infty),
$$
while it is known  (see Theorems 323, 328 in \cite{HW}) that
$$
\liminf_{n\to\infty} \frac{\phi(n)\log\log n}{n}=e^{-\gamma} \qquad \text{and}\qquad \limsup_{n\to\infty}\frac{\sigma(n)}{n\log\log n}=e^{\gamma}.
$$
In other words, while the minimal order of $\displaystyle{\frac{\phi(n)\log\log n}{n}}$ exists and is equal to $e^{-\gamma}$, the subsequence
$\displaystyle{\left(\frac{\phi(n!)\log\log n!}{n!}\right)_{n\ge 1}}$ does converge to $e^{-\gamma}$ as $n\to \infty$. In the same manner, while the maximal order of
$\displaystyle{\frac{\sigma(n)}{n\log \log n}}$ exists and is equal to $e^{\gamma}$, the subsequence $\displaystyle{\left(\frac{\sigma(n!)}{n!\log \log n!}\right)_{n\ge 1}}$ converges to $e^{\gamma}$.

Regarding the generalized sum of divisors function at factorial arguments, we have that, given any real number $\kappa>1$,
$$
\limsup_{n\to\infty}\frac{\sigma_\kappa(n)}{n^\kappa}=\zeta(\kappa),
$$
while we know from Theorem \ref{thm:sigmak} that, for $\kappa>1$,
$$
\sigma_\kappa(n!)\sim (n!)^\kappa\,\zeta(\kappa) \qquad (n\to\infty).
$$

\section{Proofs of the main results} \label{sec:proofs}

\subsection{Useful tools}

One of our main tools will be de Polignac's formula (attributed by L.E. Dickson \cite{dickson} to Legendre)
\begin{equation} \label{eq:polignac}
n! = \prod_{p\le n} p^{\alpha_p(n)}, \mbox{ where } \alpha_p(n) =   \sum_{i=1}^\infty \left\lfloor \frac n{p^i} \right\rfloor =
 \sum_{i=1}^{\lfloor\log_p n\rfloor} \left\lfloor \frac n{p^i} \right\rfloor.
\end{equation}
This formula can be somewhat simplified. Indeed, it is easy to show that
\begin{equation} \label{eq:polignac-2}
\sum_{i=1}^\infty \left\lfloor \frac n{p^i} \right\rfloor = \frac{n-s_p(n)}{p-1},
\end{equation}
where $s_p(n)$ stands for the sum of the digits of $n$ in base $p$.
For a proof of (\ref{eq:polignac-2}), see Problem 438 in the book of De Koninck and Mercier \cite{1001}.

We will also rely on classical analytic number theory results. In particular, setting $\text{Li}(x):=\int_{2}^x\frac{dt}{\log t}$, we will be using the prime number theorem in the form
\begin{equation} \label{eq:pnt}
\pi(x)=\text{Li}(x)+O(x\exp(-\sqrt{\log x})),
\end{equation}
a consequence of Theorem 12.2 in the book of Ivi\'c \cite{ivic1}.
We shall also at times use the prime number theorem in the simpler form
\begin{equation} \label{eq:pntweak}
    \pi(N)=\frac{N}{\log N}\Big(1+O\Big(\frac{1}{\log N}\Big)\Big),
\end{equation}
which can be equivalently stated as
\begin{equation} \label{eq:thetapnt}
    \theta(N) = N + O \Big( \frac N{\log N} \Big),
\end{equation}
where $\theta$ is the Chebyshev function
\begin{equation} \label{eq:theta}
\theta(N):=\sum_{p\le N} \log p.
\end{equation}

From the prime number theorem, one easily deduces the following asymptotic formula for the sum of the primes up to $N$:
\begin{equation} \label{eq:sumprimes}
\sum_{p\leq N}p = \frac{N^2}{2\log N} \left(1+O\left(\frac 1{\log N}\right)\right).
\end{equation}

We will also need Mertens' theorem in its strong form obtained by Dusart \cite{pd}:
\begin{equation} \label{eq:mertens}
\prod_{p\le N} \Big(1 - \frac 1p \Big) = \frac{e^{-\gamma}}{\log N} \Big(1 + O \Big( \frac 1{\log^3 N} \Big) \Big).
\end{equation}

\subsection{Proofs of Theorems \ref{thm:add1}, \ref{thm:add2}, and \ref{thm:mult1}} \label{sec:proofs1}
Let $f$ be an additive function and $g$ a multiplicative function. Recall that we say that $f$ is \textit{completely} additive (resp. $g$ is \textit{completely} multiplicative)  if $f(nm)=f(n)+f(m)$ (resp., $g(nm)=g(n)g(m)$) for any positive integers $n$ and $m$. Also, $f$ is \textit{strongly} additive (resp. $g$ is \textit{strongly} multiplicative) if for any prime $p$ and positive integer $a$, $f(p^a)=f(p)$ (resp. $g(p^a)=g(p)$).

Note that $\omega$ and $\beta$ are strongly additive, whereas $\Omega$ and $B$ are completely additive.
\vskip 5pt

Let us first state the following easily established proposition.

\begin{prop} \label{prop:1}
Let $g$ a strongly additive (resp. strongly multiplicative) function. Then $g(n!)=\sum_{m\le n} g(m)$ (resp. $g(n!)=\prod_{p\leq n}g(p)$).
\end{prop}

The first identity in Theorem \ref{thm:add1} is a consequence of Proposition \ref{prop:1}. The proof of the second identity, namely $\Omega(n!)=S_\Omega(n)$, 
simply follows from the fact that $\Omega(n)$ is a completely additive function and therefore that 
$$\Omega(n!)= \sum_{m\le n} \Omega(m) = S_\Omega (n).$$

The last two estimates in Theorem \ref{thm:add1} follow respectively from estimates \eqref{eq:pntweak} and \eqref{eq:Omega}.

The first identity in
Theorem \ref{thm:add2}  follows from Proposition \ref{prop:1}.

The proof of the second one, namely of the identity $B(n!)=S_B(n)$, is immediate as it follows from the fact that $B(n)$ is completely additive.

Estimates \eqref{eq:th23a} and \eqref{eq:th23b} are immediate consequences of \eqref{eq:BB1} and \eqref{eq:sumprimes}.

The proof of estimate (\ref{eq:th23c}) goes as follows.
Using de Polignac's formula \eqref{eq:polignac} and the fact that $B_1$ is additive, we have
\begin{eqnarray} \label{eq:preuve-c} \nonumber 
B_1(n!) & = & B_1 \left( \prod_{p\le n} p^{\alpha_p(n)} \right) = \sum_{p\le n} p^{\alpha_p(n)} \\
& = & 2^{\alpha_2(n)} + \sum_{3\le p \le n} p^{\alpha_p(n)} = 2^{n-s_2(n)} + \sum_{3\le p \le n} p^{\frac{n-s_p(n)}{p-1}} ,
\end{eqnarray}
where we made use of formula \eqref{eq:polignac-2}.
First observe that 
\begin{equation} \label{eq:preuve-c2}
\sum_{3\le p \le n} p^{\frac{n-s_p(n)}{p-1}} < \sum_{3\le p \le n} p^{ n/(p-1)} = (\sqrt 3)^n \left( 1+ \sum_{5\le p\le n} \left( \frac{p^{1/(p-1)}}{\sqrt 3} \right)^n \right).
\end{equation}
Since $\displaystyle{\frac{p^{1/(p-1)}}{\sqrt 3}< \frac 9{10}}$ for each prime $p\ge 5$, it follows that
$$\sum_{5\le p\le n} \left( \frac{p^{1/(p-1)}}{\sqrt 3} \right)^n < \left(\frac 9{10} \right)^n \cdot \pi(n) < 2 \qquad (n\ge 5),$$
an inequality  which substituted in \eqref{eq:preuve-c2} gives
\begin{equation} \label{eq:vieux}
\sum_{3\le p \le n} p^{\frac{n-s_p(n)}{p-1}} < (\sqrt 3)^n (1+2) \ll (\sqrt 3)^n,
\end{equation}
which inserted in \eqref{eq:preuve-c} completes the proof of estimate (\ref{eq:th23c}).

Observing that both $\gamma(n)$ and $\phi(n)/n$ are strongly multiplicative functions, the identities in \eqref{eq:multfuncest} follow from  Proposition \ref{prop:1}. Moreover, using the identity
$\log\gamma(n!)=\sum_{p\leq n}\log p$, which by the prime number theorem in the form \eqref{eq:thetapnt} is asymptotic to $n$ as $n\to\infty$, proves the first identity in \eqref{eq:multfuncas}. The last estimate in \eqref{eq:multfuncas} follows from Mertens' theorem given through estimate \eqref{eq:mertens}. Gathering these observations, the proof of Theorem \ref{thm:mult1} is complete.

\subsection{Proofs of Theorems \ref{thm:3}, \ref{thm:2}, and \ref{thm:6}} \label{sec:pfsum}
Note that \eqref{eq:as3} is an immediate consequence of \eqref{eq:ivicprime}.
The other estimates in Theorem \ref{thm:3} follow at once from standard analytic number theory techniques, including partial summation which we will use repetitively without further mention.
For instance, it follows from Theorem \ref{thm:add1} and \eqref{eq:pntweak} that
$$
\sum_{2\leq n\leq N}\omega(n!)=\sum_{2\leq n\leq N}\frac{n}{\log n}\Big(1+O\Big(\frac{1}{\log n}\Big)\Big)=\frac{N^2}{2\log N}\Big(1+O\Big(\frac{1}{\log N}\Big)\Big),
$$
which proves \eqref{eq:as1}.

We skip the proof of estimate \eqref{eq:as2} since it can be obtained in a similar manner.

The proofs of  \eqref{eq:as3} and  \eqref{eq:as4} are similar. Hence, we only prove  \eqref{eq:as4}.  By Theorem \ref{thm:add1}, we have
\begin{eqnarray*} \label{eq:pf03} \nonumber
    \sum_{2\leq n\leq N} \frac{1}{\Omega(n!)}&=&\sum_{2\leq n\leq N}\frac{1}{n\big(\log \log n +  c_1 + c_2
      + O \big( \frac 1{\log n} \big) \big)}  \\ \nonumber
    &=& \sum_{2\leq n\leq N}\frac{1}{n\log\log n\big(1+ O \big( \frac 1{\log\log n} \big) \big)}  \\
    &=& \sum_{2\leq n\leq N}\frac{1}{n\log\log n}\Big(1+O\Big(\frac{1}{\log\log n}\Big)\Big).
\end{eqnarray*}
The main term is
\begin{equation*}
\sum_{2\leq n\leq N}\frac{1}{n\log\log n} = \frac{\log N+O(1)}{\log\log N} - \int_2^N\frac{\log t+O(1)}{t\log t(\log\log t)^2}\,dt.
\end{equation*}
Integrating by parts, this last integral is
\begin{align*}
\ll  \int_2^N \frac{dt}{t(\log\log t)^2} =\frac{\log N}{(\log\log N)^2} + O(1) -\int_2^N\frac{dt}{t(\log\log t)^3} = O\Big(\frac{\log N}{(\log\log N)^2}\Big).
\end{align*}
The same upper bound holds for the remaining sum in the error term since
\begin{align*}
\sum_{2\leq n\leq N}\frac{1}{n(\log\log n)^2} =
\int_2^N \frac{dt}{t(\log\log t)^2} + O(1).
\end{align*}
Combining these estimates yields \eqref{eq:as4}.

Finally, combining Theorem \ref{thm:add1} with \eqref{eq:Omega} and \eqref{eq:pntweak} gives
\begin{align*}
    \sum_{2\leq n\leq N}\frac{\Omega(n!)}{\omega(n!)} &= \sum_{2\leq n\leq N} \frac{\log n\big(\log \log n + c_1+c_2 + O \big( \frac 1{\log n} \big) \big)}{1+O\big(\frac{1}{\log n}\big)} \\
    &=\sum_{2\leq n\leq N} \log n\Big(\log \log n +  c_1+c_2 + O \Big( \frac 1{\log n} \Big) \Big)\Big(1+O\Big(\frac{1}{\log n}\Big)\Big) \\
    &= \sum_{2\leq n\leq N}\log n\log\log n + (c_1+c_2)\sum_{2\leq n\leq N}\log n + O\Big(\sum_{2\leq n\leq N}\log\log n\Big) \\
    &= S_1+S_2+S_3,
\end{align*}
say. Clearly, $
S_2=(c_1+c_2)\,N\log N + O(N)
$ by \eqref{eq:Stirling}, and $S_3=O(N\log\log N)$.
Also, from \eqref{eq:Stirling}, we obtain
\begin{align*}
S_1 &= (N\log N + O(N))\log\log N - \int_2^N\frac{(t\log t + O(t))}{t\log t}\,dt \\
&= (N\log N + O(N))\log\log N + O(N) =N\log N\log\log N + O(N\log\log N).
\end{align*}
Estimate \eqref{eq:as5} then follows.

We now prove Theorem \ref{thm:2}. From Theorem \ref{thm:add2}, we know that
$$
\sum_{2\leq n\leq N}\beta(n!) = \sum_{2\leq n\leq N}\frac{n^2}{2\log n} \left(1 + O \left( \frac 1{\log n} \right) \right).
$$
Approximating the sum by an integral and using integration by parts, we find that
$$
\sum_{2\leq n\leq N}\frac{n^2}{2\log n} = \int_2^N \frac{t^2}{2\log t}\,dt +O(1) = \frac{N^3}{6\log N}+O(1)  .
$$
Combining these last two asymptotic formulas proves \eqref{eq:al1}.
The proof of \eqref{eq:al2} is along the same lines.
Estimate  \eqref{eq:al3} is a consequence of the two inequalities
\begin{eqnarray*}
\sum_{n\le N} B_1(n!) & < &  \sum_{n\le N} 4\cdot 2^n = 2^{N+3}, \\
 \sum_{n\le N} B_1(n!) & > & \sum_{2\le n \le N} 2^{n- \frac{\log n}{\log 2} -1} \gg  2^{N+O(\log N)},
\end{eqnarray*}
where in the first set of inequalities we used the fact that $B_1(n!)< 2^n + 3(\sqrt 3)^n < 4 \cdot 2^n$, a consequence of \eqref{eq:preuve-c} and  \eqref{eq:vieux}, while in the second set of inequalities we used the fact that
$$s_2(n) = \left\lfloor \frac{\log n}{\log 2} \right\rfloor +1 \le  \frac{\log n}{\log 2} +1.$$  

Finally, using estimate \eqref{eq:th23a}, we have that
$$B_1(n!) \ge B(n!) \ge \beta(n!)  \gg \frac{n^2}{\log n},$$
implying that if $f$ stands for any of the functions $\beta$, $B$ and $B_1$, we obtain that
$$
\sum_{n=2}^\infty \frac{1}{f(n!)}\ll \sum_{n=2}^\infty \frac{\log n}{n^2} < \sum_{n=2}^\infty \frac 1{n^{3/2}}<\infty,
$$
which completes the proof of Theorem \ref{thm:2}.

Before proceeding with the proof of Theorem \ref{thm:6}, we give a proof of  Proposition\,\ref{prop:sigphi}.

\begin{proof}[Proof of Proposition \ref{prop:sigphi}]

First observe that for any real number $s>1$,
\begin{eqnarray*}
\sum_{n=1}^\infty \frac{\sigma(n)/\phi(n)}{n^s} & = &  \prod_p \Big( 1 + \frac{\sigma(p)/\phi(p)}{p^s}
+ \frac{\sigma(p^2)/\phi(p^2)}{p^{2s}} + \cdots \Big) \\
& = & \zeta(s) \prod_p \Big( 1- \frac 1{p^s} \Big) \\
& & \qquad \times \Big( 1 + \frac{(p+1)/(p-1)}{p^s}
+ \frac{(p^2+p+1)/(p(p-1))}{p^{2s}} + \cdots \Big) \\
& = & \zeta(s) F(s),
\end{eqnarray*}
say. It is clear that
\begin{eqnarray*}
F(s) & = & \prod_p \Big( 1 + \left( \frac{(p+1)/(p-1)}{p^s} - \frac 1{p^s} \right)\\
 & & \qquad + \left( \frac{(p^2+p+1)/(p(p-1))}{p^{2s}} - \frac{(p+1)/(p-1)}{p^{2s}} \right) + \cdots \Big).
\end{eqnarray*}
Since we intend to use Wintner's theorem (already mentioned in Subsection \ref{subsec:mult-fac}, as well as its refinement given in Problem 6.3 of \cite{DL}), we need to check that $F(s)$ converges absolutely at $s=\frac 12 + \delta$ and verify that indeed $F(1)=d_1$. One can easily check that the first of these two conditions is satisfied.
To verify the second condition, observe that
\begin{eqnarray*}
F(1) & = & \prod_p \Big(1 + \frac 2{p(p-1)} + \frac 1{p^3(p-1)} + \frac 1{p^5(p-1)} + \cdots  \Big)\\
& = & \prod_p \Big(1 + \frac 2{p(p-1)} + \frac 1{p(p+1)(p-1)^2} \Big) =  \prod_p \Big(1 + \frac{2p^2-1}{p(p+1)(p-1)^2} \Big).
\end{eqnarray*}
We can therefore apply Wintner's theorem, thereby establishing our first claim.

To establish the second claim, one can proceed as above and in the end obtain that
$\displaystyle{\sum_{n=1}^\infty \frac{\phi(n)/\sigma(n)}{n^s} =  \zeta(s) G(s)}$,
with
\begin{eqnarray*}
G(1) & = & \prod_p \Big( 1 - \frac 2{p(p+1)} - \frac{p-1}{p^2} \Big( \frac 1{(p+1)(p^2+p+1)} \\
& & \qquad \qquad + \frac 1{(p^2+p+1)(p^3+p^2+p+1)} + \cdots \Big) \Big),
\end{eqnarray*}
thus completing the proof of our second assertion.
\end{proof}

We now prove  Theorem \ref{thm:6}.
To prove \eqref{eq:m1}, we will show that
$$
\sum_{n\le N} \gamma(n!) = e^{\theta(N)+O(\log N)},
$$
where $\theta$ is the Chebyshev function defined in \eqref{eq:theta}, since then the result will follow from the prime number theorem in the form \eqref{eq:thetapnt}. From here on, we let $p_i$ stand for the $i$-th prime. Now, let $r$ be the unique integer satisfying $p_r \le N < p_{r+1}$. Observing that for all $n\in [p_i,p_{i+1})$, the term $\gamma(n!)$ remains unchanged and in fact is equal to $\prod_{p\le p_i}p$, one can easily see that
\begin{equation}\label{eq:21}
\sum_{n\le N} \gamma(n!) = \sum_{i=1}^{r-1} (p_{i+1}-p_i) \prod_{p\le p_i} p + (N-p_r) \prod_{p\le p_r} p.
\end{equation}
Using Bertrand's postulate, we have that $p_{i+1}-p_i<p_i$ and
therefore
\begin{equation}\label{eq:31}
\sum_{i=1}^{r-1} (p_{i+1}-p_i) \prod_{p\le p_i} p < \sum_{i=1}^{r-1} p_i \prod_{p\le p_i} p < \sum_{i=1}^{r-1} \prod_{p\le p_{i+1}}p
< r \prod_{p\le p_r} p.
\end{equation}
Combining \eqref{eq:21} and \eqref{eq:31}, we may write that
\begin{equation}\label{eq:41}
(N-p_r) \prod_{p\le p_r} p < \sum_{n\le N} \gamma(n!) < r \prod_{p\le p_r} p + (N-p_r) \prod_{p\le p_r} p= (N-p_r+r)\prod_{p\le p_r} p.
\end{equation}
From the definition of $\theta(N)$,  we have
\begin{equation}\label{eq:51}
\prod_{p\le p_r} p = \prod_{p\le N} p= e^{\theta(N)}.
\end{equation}
On the other hand, it is clear that
$N-p_r<N-p_r+r<N$, implying that  both $N-p_r$ and $N-p_r+r$ are $e^{O(\log N)}$. Combining this observation with \eqref{eq:41} and \eqref{eq:51}, it is immediate that
$$\sum_{n\le N} \gamma(n!) = e^{\theta(N)} e^{O(\log N)} = e^{\theta(N)+O(\log N)},$$
thus completing the proof of \eqref{eq:m1}.

To prove \eqref{eq:m2}, first observe that
\begin{align*}
\sum_{2\leq n\leq N}\frac{n!}{\log n}&= \frac{N!}{\log N}\Big(1+\frac{1}{N}\frac{\log N}{\log (N-1)}+\frac{1}{N(N-1)}\frac{\log(N-1)}{\log(N-2)}+\cdots \Big) \\
&= \frac{N!}{\log N}\Big(1+O\Big(\frac{1}{N}\Big)\Big),
\end{align*}
and, similarly,
$$
\sum_{2\leq n\leq N}\frac{n!}{\log^4 n} = \frac{N!}{\log^4 N} \Big(1+O\Big(\frac{1}{N}\Big)\Big).
$$
Combining these estimates with Theorem \ref{thm:mult1} and \eqref{eq:mertens}, we obtain
$$
\sum_{2\leq n\leq N}\phi(n!) = \sum_{2\leq n\le N}n!\frac{e^{-\gamma}}{\log n}\Big(1+O\Big(\frac{1}{\log^3n}\Big)\Big)
= e^{-\gamma}\frac{N!}{\log N}\Big(1+O\Big(\frac{1}{\log^3N}\Big)\Big),
$$
which proves \eqref{eq:m2}.

The proof of \eqref{eq:m4} is similar. On the other hand, \eqref{eq:m3} follows from partial summation after using the estimate
$$
\log d(n!) = c_0\frac{n}{\log n} + O\Big(\frac{n}{\log^2 n}\Big),
$$
a consequence of \eqref{eq:dn!4}.

Using the first estimate in the statement of Theorem \ref{thm:sigmak}, the second formula in the statement of Theorem \ref{thm:mult1} and estimate \eqref{eq:mertens} (Mertens' theorem), we may write that
\begin{eqnarray*}
\frac{\sigma(n!)}{\phi(n!)}
& = & e^{2\gamma} \log^2n \Big(1 + O \Big( \frac 1{\log^3 n} \Big) \Big)\Big(1 + O \Big( \frac 1{\log^3 n} \Big) \Big)\\
& = & e^{2\gamma} \log^2n \Big(1 + O \Big( \frac 1{\log^3 n} \Big) \Big).
\end{eqnarray*}
Summing this last expression over all integers $n\in [2,N]$, we obtain
\begin{eqnarray*}
\sum_{2\le n\le N} \frac{\sigma(n!)}{\phi(n!)} & = & e^{2\gamma} \sum_{2\le n\le N} \log^2 n  \Big(1 + O \Big( \frac 1{\log^3 n} \Big) \Big) \\
& = & e^{2\gamma} \Big(1 + O \Big( \frac 1{\log^3 N} \Big) \Big) \int_2^N \log^2 t\,dt \\
& = & e^{2\gamma} (N\log^2 N -2N\log N +2N )  \Big(1 + O \Big( \frac 1{\log^3 N} \Big) \Big),
\end{eqnarray*}
thus establishing \eqref{eq:m5}.

The proof of \eqref{eq:m6} follows along the same lines, thus completing the proof of Theorem \ref{thm:6}.

\subsection{Proofs of Theorems \ref{thm:dk}, \ref{thm:central}, and \ref{thm:rho}}
Our approach is at first similar to the one used in \cite{egip} for  $k=2$.
For a general $k\ge 2$, recall that $d_k(p^a)=\binom{k+a-1}{k-1}$, so that, in light of \eqref{eq:polignac}, we have
$$
d_k(n!)=\prod_{p\leq n}\binom{k+\alpha_p(n)-1}{k-1}.
$$
It is more convenient to work with $\log d_k(n!)$, splitting the resulting sum at $p=n^{3/4}$. Since $$\binom{k+\alpha_p(n)-1}{k-1}\ll_k \alpha_p^{k-1}(n)$$ and $\alpha_p(n)<\frac{n}{p-1}$ (this last inequality following from \eqref{eq:polignac-2}), we get
$$
\sum_{p\leq n^{3/4}}\log\binom{k+\alpha_p(n)-1}{k-1} \ll_k \sum_{p\leq n^{3/4}}\log\Big(\frac{n}{p-1}\Big) \ll_k \log n\sum_{p\leq n^{3/4}} 1 \ll_k n^{3/4},
$$
where in this last inequality we used \eqref{eq:pntweak}.

For $p\in (n^{3/4}, n]$, we have ${\alpha_p(n)=\lfloor n/p\rfloor}$. We therefore have
\begin{equation} \label{eq:pfdk1}
\sum_{n^{3/4}<p\leq n}\log\binom{k+\alpha_p(n)-1}{k-1} =\int_{n^{3/4}}^n\log\binom{k+\left\lfloor \frac{n}{x} \right\rfloor-1}{k-1} \,d\pi(x).
\end{equation}
Writing $\pi(x)=\text{Li}(x)+E(x)$ where $E(x)=O(x\exp(-\sqrt{\log x}))$ by \eqref{eq:pnt}, partial summation on the right-hand side of \eqref{eq:pfdk1} gives
\begin{equation*}
I:=\int_{n^{3/4}}^n\log\binom{k+\left\lfloor \frac{n}{x} \right\rfloor-1}{k-1}\,\frac{dx}{\log x}
\end{equation*}
as the main term, while the contribution of the error term is
$$
\int_{n^{3/4}}^n\log\binom{k+\left\lfloor \frac{n}{x} \right\rfloor-1}{k-1} \,dE(x)\ll_k  \log n\int_{n^{3/4}}^n dE(x) \ll_k n\log n\exp(-\sqrt{\log n}).
$$

To estimate the remaining integral $I$, we start with the substitution $t=n/x$ so that
$$
I = n\int_1^{n^{1/4}}\log\binom{k+\left\lfloor t \right\rfloor-1}{k-1} \frac{dt}{t^2\log(n/t)}.
$$
Since $1\leq t\leq n^{1/4}$, we have that for any fixed integer $M\geq 0$,
$$
\frac{1}{\log(n/t)}=\frac{1}{\log n\big(1-\frac{\log t}{\log n}\big)} = \frac{1}{\log n}\left\{ \sum_{\ell=0}^M\Big(\frac{\log t}{\log n}\Big)^\ell + O\Big(\frac{\log t}{\log n}\Big)^{M+1}\right\},
$$
from which we get
\begin{equation} \label{eq:pfdk2}
    I =\frac{n}{\log n}\int_{1}^{n^{1/4}}\log\binom{k+\left\lfloor t \right\rfloor-1}{k-1}\left\{ \sum_{\ell=0}^M\Big(\frac{\log t}{\log n}\Big)^\ell + O\Big(\frac{\log t}{\log n}\Big)^{M+1}\right\}\frac{dt}{t^2}.
\end{equation}
Since $\displaystyle{\int_{n^{1/4}}^\infty \log\binom{k+\left\lfloor t \right\rfloor-1}{k-1} (\log t)^{\ell}\frac{dt}{t^2}}$ converges and since
$$
\int_{n^{1/4}}^\infty \log\binom{k+\left\lfloor t \right\rfloor-1}{k-1} (\log t)^{\ell}\frac{dt}{t^2} \ll_k \int_{n^{1/4}}^{\infty} \frac{\log^{\ell+1}t}{t^2}\,dt \ll_k \frac{\log^{\ell+1}n}{n^{1/4}}
$$
for  $\ell = 0,1,\ldots,M$,
the error term in \eqref{eq:pfdk2} is
$$
O\Big(\frac{n}{\log^{M+2} n}\int_1^{n^{1/4}} \log\binom{k+\left\lfloor t \right\rfloor-1}{k-1} \log^{M+1} t\,\frac{dt}{t^2} \Big) = O\Big(\frac{n}{\log^{M+2} n}\Big),
$$
while the main term is
\begin{equation*}
\begin{split}
  \frac{n}{\log n} \sum_{\ell=0}^{M}&\frac{1}{(\log n)^\ell} \left\{\int_1^{\infty} \log\binom{k+\left\lfloor t \right\rfloor-1}{k-1} \log^\ell t\,\frac{dt}{t^2} + O\Big(\frac{\log^{\ell+1}n}{n^{1/4}} \Big)\right\} \\
&=  \frac{n}{\log n}\sum_{\ell=0}^{M}\frac{c_\ell^{(k)}}{(\log n)^\ell} + O(n^{3/4}),
\end{split}
\end{equation*}
where
\begin{equation} \label{eq:ckint}
    c_\ell^{(k)}= \int_1^{\infty} \log\binom{k+\left\lfloor t \right\rfloor-1}{k-1} \log^\ell t\,\frac{dt}{t^2}.
\end{equation}
This proves Theorem \ref{thm:dk} with the constants given in \eqref{eq:ckint}. We note that similar constants appeared in \cite{dkr,nas}, where the authors worked with integrals involving fractional parts. Here we use the same approach as the one used in these papers to show that the constants given in \eqref{eq:ckint} can be expressed as in \eqref{eq:clk}. Indeed, we have
$$c_{\ell}^{(k)}= \sum_{s=1}^\infty \int_s^{s+1}\log\binom{k+s-1}{k-1}\log^\ell t \,\frac{dt}{t^2}
   =\sum_{s=1}^\infty \log\binom{k+s-1}{k-1} I_\ell(s),$$
where
$\displaystyle{I_\ell(s):=\int_s^{s+1}\frac{\log^\ell t}{t^2}\,dt}$.
Integrating by parts, we get
$$
I_\ell(s)=\left.-\frac{\log^\ell t}{t}\right|_s^{s+1} + \ell\int_s^{s+1}\frac{\log^{\ell-1} t}{t^2}\,dt = \frac{\log^\ell s}{s}-\frac{\log^\ell (s+1)}{s+1} + \ell\cdot I_{\ell-1}(s).
$$
Repeating this $\ell-1$ times and observing that $I_0(s)=\int_s^{s+1}t^{-2}\,dt=\frac{1}{s(s+1)}$, we get
\begin{align*}
 I_\ell(s) &=\frac{\log^\ell s}{s}-\frac{\log^\ell (s+1)}{s+1} \\
 & \quad + \sum_{i=1}^{\ell-1} \ell(\ell-1)\cdots (\ell-i+1)\Big(\frac{\log^{\ell-i}s}{s}
    -  \frac{\log^{\ell-i}(s+1)}{s+1} \Big) + \frac{\ell!}{s(s+1)} \\
    &= \sum_{i=0}^{\ell}\frac{\ell!}{(\ell-i)!}\Big( \frac{\log^{\ell-i}s}{s} - \frac{\log^{\ell-i}(s+1)}{s+1} \Big),
\end{align*}
 from which we obtain that
\begin{align*}
 c_\ell^{(k)}=  &\sum_{s=1}^\infty \log\binom{k+s-1}{k-1} \sum_{i=0}^{\ell}\frac{\ell!}{(\ell-i)!}\Big( \frac{\log^{\ell-i}s}{s} - \frac{\log^{\ell-i}(s+1)}{s+1} \Big) \\
   &= \sum_{i=0}^{\ell}\frac{\ell!}{(\ell-i)!}\sum_{s=1}^\infty \log\binom{k+s-1}{k-1}\Big( \frac{\log^{\ell-i}s}{s} - \frac{\log^{\ell-i}(s+1)}{s+1} \Big) \\
   &= \sum_{i=0}^{\ell}\frac{\ell!}{(\ell-i)!}\sum_{s=1}^\infty\Big(\log\binom{k+s-1}{k-1}-\log\binom{k+s-2}{k-1}\Big)\frac{\log^{\ell-i}s}{s} \\
   &= \sum_{i=0}^{\ell}\frac{\ell!}{(\ell-i)!}\sum_{s=1}^\infty\log\Big(\frac{k+s-1}{s}\Big) \frac{\log^{\ell-i}s}{s}.
\end{align*}
Summing the terms backwards, we get the desired expression \eqref{eq:clk}.
Finally, observe that
$$
\sum_{s=1}^\infty \frac{\log^j s}{s}\log\Big(1+\frac{k-1}{s}\Big) = O(1)+ (k-1)\sum_{s=1}^\infty \frac{\log^j s}{s^2} + O\Big((k-1)^2\sum_{s=1}^\infty \frac{\log^j s}{s^3} \Big).
$$
Approximating the sum by an integral, we have
$$
\sum_{s=1}^\infty \frac{\log^j s}{s^2} = \int_1^\infty \frac{\log^j t}{t^2}\,dt + O(1) = \int_0^\infty t^je^{-t}\,dt + O(1) = \Gamma(j+1)+O(1),
$$
where $\Gamma(j+1)$ is
the Gamma function at $j+1$, which equals $j!$.
In the same way,
$$
\sum_{s=1}^\infty \frac{\log^j s}{s^3} = O\Big(\frac{j!}{2^{j+1}}\Big),
$$
so that
$$
c_\ell^{(k)} = \ell!\sum_{j=0}^\ell \Big(k-1+O\Big(\frac{(k-1)^2}{2^{j+1}}\Big)\Big)= (k-1)(\ell+1)! + O\big((k-1)^2(1-2^{-\ell-1})\ell!\big).
$$
This proves the last claim of Theorem \ref{thm:dk}.

Theorem \ref{thm:rho} is proved in the exact same way, but with the simpler constants
$$
e_\ell = \int_1^\infty \frac{\log\lfloor t\rfloor\log^\ell t}{t^2}\,dt \qquad (\ell\geq 0).
$$

To prove Theorem \ref{thm:central}, first observe that
$$
\binom{2n}{n} = \prod_{p\leq 2n}p^{\alpha_p\binom{2n}{n}},$$  where
$$\alpha_p\binom{2n}{n} = \sum_{j\geq 1}\Big(\left\lfloor \frac{2n}{p^j}\right\rfloor - 2\left\lfloor \frac{n}{p^j}\right\rfloor\Big).
$$
It follows that
$$
\log d_k\Big(\binom{2n}{n}\Big) = \sum_{p\leq 2n}\log\binom{k+\alpha_p\binom{2n}{n}-1}{k-1}.
$$
If $n<p\leq 2n$, then $\log\binom{k+\alpha_p\binom{2n}{n}-1}{k-1}$ reduces to $\log k$, so that
$$
\sum_{n<p\leq 2n}\log\binom{k+\alpha_p\binom{2n}{n}-1}{k-1} = (\pi(2n)-\pi(n))\log k.
$$
To deal with the remaining part, we use the same method as in the proof of Theorem \ref{thm:dk} to obtain that for any integer $M\geq 0$,
$$
\sum_{p\leq n}\log\binom{k+\alpha_p\binom{2n}{n}-1}{k-1}
=
\frac{n}{\log n}\sum_{\ell=0}^M\frac{b_\ell^{(k)}}{\log^\ell n} + O\Big(\frac{n}{\log^{M+2}n}\Big),
$$
where this time
$$
b_\ell^{(k)} = \int_1^\infty \log\binom{k+[2t]-2[t]-1}{k-1}\log^\ell t\,\frac{dt}{t^2}.
$$
Now observe that the integrand vanishes unless $n+\frac{1}{2}\leq t<n+1$ for some integer $n\geq 1$, in which case it becomes $t^{-2}\log^\ell t\log k.$
The result then follows.

\subsection{Proof of Theorem \ref{thm:sigmak}}
According to  Mertens' theorem already stated in \eqref{eq:mertens}, ${\prod_{p\leq x}(1-1/p)\sim e^{-\gamma}/\log x}$ as $x\to\infty$.
Perhaps surprisingly, the behavior of the somewhat similar product $\prod_{p\leq n}(1-p^{-\alpha_p(n)-1})$ is very different, since as we will now see it is asymptotic to $1$ as $n\to\infty$.

\begin{lem} \label{lem:An}
    For integers $n\geq 2$,
    $$
\prod_{p\le n} \Big(1 - \frac 1{p^{\alpha_p(n)+1}} \Big) = 1+O\Big(\frac{\log n}{n}\Big).
    $$
\end{lem}

\begin{proof}
First set $\displaystyle{A(n) := \prod_{p\le n} \Big(1 - \frac 1{p^{\alpha_p(n)+1}} \Big)}$.
Clearly, $A(n)<1$ for all $n\ge 2$. We will therefore focus our attention on finding a lower bound for $A(n)$ which is ``very close'' to 1.

Let $r\ge 3$ be a fixed integer (which will eventually be chosen to be large). Then,
$$A(n) = \prod_{p\le n/r} \Big(1 - \frac 1{p^{\alpha_p(n)+1}} \Big) \cdot
\prod_{n/r< p \le n} \Big(1 - \frac 1{p^{\alpha_p(n)+1}} \Big)={A'}_{r}(n) \cdot {A''}_{r}(n),$$
say.
Since, for each prime $p\le n$, we have $\alpha_p(n)\ge 1$, it follows that $\displaystyle{\frac 1{p^{\alpha_p(n)+1}}\le \frac 1{p^2}}$ and therefore that
$\displaystyle{1- \frac 1{p^{\alpha_p(n)+1}} \ge 1 - \frac 1{p^2}}$,
which means that
 \begin{equation*}
A^{''}_r(n)=\prod_{n/r < p\le n} \Big(1 - \frac 1{p^{\alpha_p(n)+1}} \Big) \ge \prod_{n/r <p\le n} \Big( 1 - \frac 1{p^2} \Big).
\end{equation*}
Since for all $y\in (0,1)$, we have $\log(1-y)>-2y$, we may write that
\begin{equation} \label{eq:6}
A^{''}_r(n) \ge   \exp\Big\{ \sum_{n/r<p\le n} \log(1-1/p^2) \Big\}> \exp\Big\{ -2 \sum_{n/r<p\le n} \frac 1{p^2}  \Big\}.
\end{equation}
Now observe that
\begin{equation} \label{eq:7}
\sum_{n/r<p\le n} \frac 1{p^2} < \sum_{p>n/r} \frac 1{p^2}< \int_{n/r}^\infty \frac{dt}{t^2} = \frac rn.
\end{equation}
Using (\ref{eq:7}) in (\ref{eq:6}), we obtain
\begin{equation} \label{eq:8}
A^{''}_r(n) > \exp \Big\{ -2 \frac rn  \Big\} = 1 + O \left( \frac rn \right).
\end{equation}

We now move to find a lower bound for $A^{'}_r(n)$. First observe that for $p\le n/r$, we have
$$\alpha_p(n) \ge \left\lfloor \frac np \right\rfloor \ge \left\lfloor r \right\rfloor = r,$$
which implies that
$\alpha_p(n)+1 \ge r+1$ and therefore that
\begin{equation} \label{eq:9}
A^{'}_r(n) \ge \prod_{p\le n/r} \Big(1 - \frac 1{p^{r+1}} \Big) > \prod_p  \Big(1 - \frac 1{p^{r+1}} \Big) = \frac 1{\zeta(r+1)}.
\end{equation}
Approximating \eqref{eq:zeta} by an integral, we have for any real $s\geq 3$,
$$\zeta(s) \le 1 + \frac 1{2^s} + \int_2^\infty t^{-s}\,dt = 1+ \frac 1{2^s} + \frac 1{2^{s-1}(s-1)}\le   1+ \frac 1{2^s} +  \frac 1{2^s} = 1+ \frac 2{2^s}. $$
Using this inequality with $s=r+1$ in \eqref{eq:9}, we find that
\begin{equation} \label{eq:11}
A^{'}_r(n)  > \frac 1{\zeta(r+1)} \ge  \frac 1{1+ 1/2^r} > 1 - \frac 1{2^r} .
\end{equation}

Combining \eqref{eq:8} with (\ref{eq:11}), while choosing $\displaystyle{r=\left\lfloor \frac{\log n}{\log 2} \right\rfloor}$, we obtain that
\begin{equation} \label{eq:12}
A(n)= A^{'}_r(n)\cdot A^{''}_r(n)  >  \Big(1 + O \Big( \frac{\log n}n \Big)\Big) \cdot \Big( 1+O\Big(\frac 1n \Big) \Big)= 1 + O \Big( \frac{\log n}n \Big).
\end{equation}
Combining (\ref{eq:12}) with our first observation to the effect that $A(n)<1$ establishes that
$\displaystyle{A(n) = 1+ O \Big( \frac{\log n}n \Big)}$, thus completing the proof of Lemma \ref{lem:An}.
\end{proof}

We can now prove Theorem \ref{thm:sigmak}.
Let $\kappa\geq 1$ and observe that for any prime power $p^a$,
\begin{equation} \label{eq:sigpower}
\sigma_\kappa(p^{a})=1+p^\kappa+\cdots + p^{a\kappa}=\frac{p^{(a+1)\kappa}-1}{p^\kappa-1} =
p^{a\kappa}\frac{1-1/p^{(a+1)\kappa}}{1-1/p^\kappa}.
\end{equation}
It follows from equations \eqref{eq:polignac} and \eqref{eq:sigpower} that
$$
\sigma_\kappa(n!)
=\prod_{p\leq n}p^{\alpha_p(n)\kappa}\frac{1-1/p^{(\alpha_p(n)+1)\kappa}}{1-1/p^\kappa}
=\frac {(n!)^\kappa}{\prod_{p\le n}(1 - \frac {1}{p^\kappa})} \cdot \prod_{p\leq n}\Big(1-\frac{1}{p^{(\alpha_p(n)+1)\kappa}}\Big).
$$
Letting $A_\kappa(n)=\prod_{p\leq n}(1-1/p^{(\alpha_p(n)+1)\kappa})$,
we see from \eqref{eq:mertens} that
$$
\sigma(n!)=n!e^\gamma\log n\Big(1+O\Big(\frac{1}{\log^3 n}\Big)\Big)\cdot A_1(n),
$$
and from \eqref{eq:zeta} that $\sigma_\kappa(n!) = (n!)^\kappa\zeta(\kappa)\cdot A_\kappa(n)$ for $\kappa>1$.
Observing that $1>A_\kappa(n)>A_1(n)=A(n)$, we  conclude that the estimate in Lemma \ref{lem:An} holds for any fixed $\kappa\geq 1$, and the result follows.

\subsection{Proofs of Theorems \ref{thm:middle-1}, \ref{thm:middle-2}, and \ref{thm:middle-3}} \label{sec:mid1}
Studying sums of the middle divisors is much easier at factorial arguments because,  according to Proposition \ref{prop:d1d2}, the values of $\rho_1(n!)$ and $\rho_2(n!)$ are close to $\sqrt{n!}$.

\begin{proof}[Proof of Proposition \ref{prop:d1d2}]
First of all, it is clear that (\ref{eq:close-2}) is an immediate consequence of (\ref{eq:close-1}). We will therefore focus on the proof of (\ref{eq:close-1}).
We start with the following construction. First recall that
\begin{equation} \label{eq:recall-n!}
n! = \prod_{p\le n} p^{\alpha_p(n)}, \quad \mbox{where } \alpha_p(n) = \frac{n-s_p(n)}{p-1},
\end{equation}
with $s_p(n)$ standing for the sum of the digits of $n$ in base $p$. Observing that $\alpha_p(n)=1$ for all primes $p\in (n/2,n]$, we may write the factorisation of $n!$ as
\begin{equation} \label{eq:facto-n!}
n! = \underbrace{2\cdots 2}_{\alpha_2(n)\,\mbox{\tiny times}} \cdot  \underbrace{3\cdots 3}_{\alpha_3(n)\,\mbox{\tiny times}} \cdot  \underbrace{5\cdots 5}_{\alpha_5(n)\,\mbox{\tiny times}}
\cdots p_{\pi(n/2)+1} \cdots p_{\pi(n)},
\end{equation}
where $p_i$ stands for the $i$-th prime, and therefore in the above $p_{\pi(n)} =P(n!)=P(n)$, the largest prime factor of $n$.

Set $w_1=w_1(n):=p_{\pi(n/2)+1} \cdots p_{\pi(n)}$ and $w_2=w_2(n)=n!/w_1$, so that
$$n! = \underbrace{2^{\alpha_2(n)} \cdot 3^{\alpha_3(n)} \cdots p_k^{\alpha_{p_k}(n)}}_{w_2} \cdot \underbrace{p_{k+1}\cdots p_{\pi(n)}}_{w_1},$$
where $k=\pi(n/2)$.

Recall the following result of R. Tijdeman \cite{tij}.
\begin{quote}
{\sl Given an infinite sequence $n_1<n_2<\cdots$ of $B$-friable numbers (or $B$-smooth numbers), there exists a positive constant $c(B)$
such that
\begin{equation} \label{eq:tij}
n_{i+1}-n_i \ll \frac{n_i}{\log^{c(B)} n_i}.
\end{equation}}
\end{quote}
It is easily seen that if $B_1<B_2$, then $c(B_2)\ge c(B_1)$. Therefore, one can assume the existence of a positive number $c$ satisfying $c\le c(B)$ for all integers $B\ge 2$. In fact, Langevin \cite{langevin} later obtained a lower bound for this constant $c$, a very small one indeed (actually somewhat smaller than $1/10^{36}$) but nevertheless effective. In the following, we will therefore assume that $c\in(0,1)$.

For convenience, let us assume that $n\ge 6$ and consider the set $S$ of all the divisors of $w_2$ which are no larger than $\sqrt{n!}/w_1$ and let $r:=\max S$. Finally, let $d_1=r \cdot w_1$. It is clear that $d_1\mid n!$ and that $d_1\le \sqrt{n!}$.

Since $w_1=w_1(n)>1$, it follows that $\sqrt{n!}$ is never an integer. Furthermore, observe that if
$$n_1<n_2<\cdots< n_k=: d_1 < \sqrt{n!} < n_{k+1}=:d_1^{+}$$
 is the list of all the positive divisors of $n!$ smaller than $\sqrt{n!}$ plus the one located immediately after $\sqrt{n!}$, then it follows from (\ref{eq:tij}) that for some positive constant\,$A$,
\begin{equation}\label{eq:d1-a}
n_{k+1}-n_k =d_1^{+} -d_1 < A \frac{d_1}{(\log d_1)^{c}}.
\end{equation}
Since the function $t/(\log t)^{c}$ is increasing for all $t\ge e$, it follows from (\ref{eq:d1-a}) that
\begin{equation}\label{eq:d1-b}
d_1^{+} -d_1 < A \frac{\sqrt{n!}}{(\log \sqrt{n!})^{c}},
\end{equation}
and therefore that
$$
d_1 > d_1^{+} - A \frac{\sqrt{n!}}{(\log \sqrt{n!})^{c}} > \sqrt{n!}  - A \frac{\sqrt{n!}}{(\log \sqrt{n!})^{c}}.
$$
Using a weak form of Stirling's formula, we have $\displaystyle{\sqrt{n!} > \left( \frac ne \right)^{n/2}}$, thereby implying that
\begin{equation}\label{eq:d1-d}
(\log \sqrt{n!})^{c} > \Big( \frac n2(\log n-1) \Big)^{c} > n^c,
\end{equation}
provided $n\geq 21$.

Using (\ref{eq:d1-d}) in (\ref{eq:d1-b}), we obtain
\begin{equation}\label{eq:d1-e}
d_1 > \sqrt{n!} - \frac{\sqrt{n!}}{n^{c}} = \sqrt{n!} \Big( 1 - \frac 1{n^c} \Big),
\end{equation}
thus completing the proof of the inequalities on the left-hand side of (\ref{eq:close-1}). Those on the right-hand side will easily follow by simply setting $d_2:=n!/d_1$; indeed, we then have
$$d_2= \frac{n!}{d_1} < \frac{n!}{(1-1/n^c)\sqrt{n!}} < \sqrt{n!} \Big(1 + \frac 2{n^c} \Big),$$
provided $n$ is sufficiently large.
\end{proof}

We will also be using the following  lemma, whose proof is immediate.
\begin{lem} \label{lem:1}
For any given integer $N\ge 1$, we have
$$
\sum_{n=1}^N n! =  N! \Big( 1 + O \Big( \frac 1N \Big) \Big)\quad \mbox{and} \quad
\sum_{n=1}^N \sqrt{n!}  =  \sqrt{N!} \Big( 1 + O \Big( \frac 1{\sqrt N} \Big) \Big).
$$
\end{lem}

We now prove Theorem \ref{thm:middle-1}.
It follows from Proposition \ref{prop:d1d2} that
\begin{equation} \label{eq:th1-a}
\sum_{n\le N} \rho_1(n!)  =  \sum_{n\le N} \sqrt{n!} \Big( 1 + O \Big( \frac 1{n^c} \Big) \Big)
   =  \sum_{n\le N} \sqrt{n!} + O \Big( \sum_{n\le N} \frac{\sqrt{n!}}{n^c} \Big).
\end{equation}
Observe that
\begin{eqnarray} \label{eq:th1-b} \nonumber
\sum_{n\le N} \frac{\sqrt{n!}}{n^c} & = & \frac{\sqrt{N!}}{N^c} \Big( 1 + \Big(\frac N{N-1}\Big)^c \frac 1{\sqrt N} +
\Big(\frac N{N-2}\Big)^c \frac 1{\sqrt{N(N-1)}} + \cdots \Big) \\
& = &     \frac{\sqrt{N!}}{N^c} \Big( 1 + O \Big( \frac 1{\sqrt N} \Big) \Big) = O \Big(  \frac{\sqrt{N!}}{N^c}  \Big).
\end{eqnarray}
Substituting (\ref{eq:th1-b}) in (\ref{eq:th1-a}), and using Lemma \ref{lem:1}, we get that
\begin{eqnarray*}
\sum_{n\le N} \rho_1(n!) & = & \sum_{n\le N} \sqrt{n!}  + O \Big(  \frac{\sqrt{N!}}{N^c}  \Big)
 =  \sqrt{N!} \Big( 1 + O \Big( \frac 1{\sqrt N} \Big) \Big) + O \Big(  \frac{\sqrt{N!}}{N^c}  \Big) \\
& = & \sqrt{N!} \Big( 1 + O \Big( \frac 1{N^c} \Big) \Big),
\end{eqnarray*}
thus proving the first estimate of Theorem \ref{thm:middle-1}. The second estimate is obtained through a similar reasoning.

We move on to Theorem \ref{thm:middle-2}.
Using Proposition \ref{prop:d1d2}, we obtain
\begin{equation} \label{eq:th2-a}
\sum_{n\le N} \frac{\rho_1(n!)}{\rho_2(n!)}  =  \sum_{n\le N} \frac{\sqrt{n!}(1+O(1/n^c))}{\sqrt{n!}(1+O(1/n^c))}
= \sum_{n\le N}\Big( 1 + O \Big( \frac 1{n^c} \Big) \Big)
 =  N + O \Big( \sum_{n\le N} \frac 1{n^c}  \Big).
\end{equation}
Approximating this last sum by an integral, we get
$$  \sum_{n\le N} \frac 1{n^c} = \int_1^N \frac 1{t^c}dt +O(1) = \frac{N^{1-c}}{1-c} +O(1) =O(N^{1-c}).$$
Using this last estimate in (\ref{eq:th2-a}), we obtain that
$$\sum_{n\le N} \frac{\rho_1(n!)}{\rho_2(n!)} = N + O(N^{1-c}),$$
thus establishing the first estimate of Theorem \ref{thm:middle-2}. The second estimate can be proved in a similar manner.

To prove Theorem \ref{thm:middle-3},
we use Proposition \ref{prop:d1d2} and Stirling's formula \eqref{eq:stirling} to obtain
\begin{eqnarray*}
\sum_{2\le n\le N} \frac{\log \rho_1(n!)}{\log \rho_2(n!)} & = &
\sum_{2\le n\le N} \frac{ \log( \sqrt{n!} (1+O(1/n^c))) }{ \log( \sqrt{n!} (1+O(1/n^c))) } \\
& = & \sum_{2\le n\le N} \frac{ \log \sqrt{n!} + \log(1+O(1/n^c))}{ \log \sqrt{n!} + \log(1+O(1/n^c))} \\
& = & \sum_{2\le n\le N} \frac{\log \sqrt{n!} + O(1/n^c)}{\log \sqrt{n!} + O(1/n^c)} =
\sum_{2\le n\le N} \frac{1 + O \Big( \frac 1{n^c \log\sqrt{n!}} \Big)}{1 + O \Big( \frac 1{n^c \log\sqrt{n!}} \Big)}\\
& = & \sum_{2\le n\le N} \Big( 1 + O \Big( \frac 1{n^{c+1}\log n} \Big) \Big) =  N + O \Big( \sum_{2\le n\le N} \frac 1{n^{c+1}\log n}   \Big)\\
& = & N + O(1),
\end{eqnarray*}
since the series $\displaystyle{\sum_{n=2}^\infty \frac 1{n^{c+1}\log n}}$ converges. This establishes the first estimate in Theorem \ref{thm:middle-3}. The second estimate is proved similarly.

\section{An analogue of Chowla's conjecture for factorial arguments} \label{sec:chowla}

Also of interest is the behavior of the quotient $\displaystyle{ \frac{f(n!)}{f(n-1)!}}$ for various arithmetic functions $f$.
For instance, in \cite{egip}, the authors showed that
$$
\frac{d(n!)}{d((n-1)!)}=1+\frac{P(n)}{n}+O\Big(\frac{1}{n^{1/2}}\Big),
$$
where $P(n)$ is the largest prime factor of $n$. Such ratios of arithmetic functions are easier to manage for some large classes of arithmetic functions. For instance, if $f$ is completely additive, then
    $$
\frac{f(n!)}{f((n-1)!)} = 1+\frac{f(n)}{S_f(n-1)},
    $$
    while if $f$ is strongly additive,
    $$
\frac{f(n!)}{f((n-1)!)} = \begin{cases}
1+\frac{f(n)}{\sum_{p\leq n-1}f(p)} &\text{if } n \text{ is prime}, \\
1 &\text{otherwise}.
    \end{cases}
    $$
In particular, the additive functions $\omega$, $\Omega$, and $\beta$ are such that the ratio of their consecutive values at factorial arguments is $\sim 1$ as $n\to\infty$.

On the other hand, for the multiplicative functions $\gamma$ and $\phi$, the results are more interesting.  This is because, for every completely multiplicative function $f$, we have
$$\frac{f(n!)}{f((n-1)!)}=f(n),$$ and
for every strongly multiplicative function $f$, we have
 $$\frac{f(n!)}{f((n-1)!)}=\begin{cases}
f(n) &\text{if } n \text{ is prime}, \\
1 &\text{otherwise}.
    \end{cases}$$
Because  $\gamma(n)$ and $\phi(n)/n$   are strongly multiplicative functions, we obtain that
$$
\frac{\gamma(n!)}{\gamma((n-1)!)}=\begin{cases}
n &\text{if } n \text{ is prime}, \\
1 &\text{otherwise},
    \end{cases}
$$
and
$$
\frac{\phi(n!)}{\phi((n-1)!}=\begin{cases}
n-1 &\text{if } n \text{ is prime}, \\
n &\text{otherwise}.
    \end{cases}
$$

Let us now consider the {\it Liouville function}  $\lambda(n):=(-1)^{\Omega(n)}$.
A famous conjecture due to Chowla \cite{chowla} in its simplest form can be stated as follows.
\begin{quote}
{\bf Conjecture (Chowla)} As $N\to \infty$,
$$\sum_{n\le N} \lambda(n)\lambda(n+1) =o(N) .$$
\end{quote}

Interestingly, the analogous form of Chowla's conjecture at factorial arguments is true. Indeed, we have the following result.
\begin{thm} \label{thm:chowla}
As $N\to \infty$,
$$\sum_{n\le N} \lambda(n!)\lambda((n+1)!) =o(N) .$$
\end{thm}

\begin{proof}
The proof is quite straightforward and in fact, as we will see, it is a consequence of the prime number theorem.
We will show that
\begin{equation} \label{eq:chowla-1}
\sum_{2\le n \le N} \lambda((n-1)!)\lambda(n!) =o(N) \qquad (N\to \infty).
\end{equation}
Clearly,
\begin{eqnarray*}
\lambda((n-1)!)\lambda(n!) & = & (-1)^{ \sum_{k\le n-1} \Omega(k) +  \sum_{k\le n} \Omega(k)}\\
& = & (-1)^{ 2\sum_{k\le n-1} \Omega(k) +   \Omega(n)}
= (-1)^{\Omega(n)} = \lambda(n).
\end{eqnarray*}
Therefore, in order to prove (\ref{eq:chowla-1}), we only need to prove that
\begin{equation} \label{eq:chowla-2}
\sum_{2\le n\le N} \lambda(n) =o(N) \qquad (N\to \infty).
\end{equation}
Recall that if we let $\mu$ stand for the M\"obius function, then  the prime number theorem implies that
\begin{equation} \label{eq:mux}
\sum_{2\le n\le N} \mu(n) =o(N) \qquad (N\to \infty)
\end{equation}
(see for instance Theorem 5.3 in \cite{DL}). Now, one can easily verify the identity
\begin{equation} \label{eq:ident}
\lambda(n)= \sum_{d^2k=n} \mu(k) \qquad (n\ge 1).
\end{equation}
Setting $M(N):=\sum_{n\le N}\mu(n)$, estimate \eqref{eq:mux} implies that $M(N)=o(N)$ as $N\to \infty$ and therefore that, given any arbitrarily small $\varepsilon>0$, there exists a large number $N_0$ (which we can assume to be sufficiently large so that $1/N_0<\varepsilon/2$) such that
\begin{equation} \label{eq:apply}
M(N)< \frac{\varepsilon}4 N \quad \mbox{  for all }N\ge N_0.
\end{equation}
Then, for any integer  $N>N_0$, in light of \eqref{eq:ident},  we have
\begin{eqnarray} \label{eq:liou} \nonumber
\frac 1N \sum_{n=1}^N \lambda(n) & = &  \frac 1N \sum_{d=1}^N \sum_{1\le k \le N/d^2} \mu(k) =  \frac 1N \sum_{d=1}^N M(N/d^2)\\
& = & \frac 1N \sum_{d=1}^{N_0} M(N/d^2) + \frac 1N \sum_{d=N_0+1}^N M(N/d^2) = S_1 +S_2,
\end{eqnarray}
say. Choosing $N>N_0^3$, for each $d\le N_0$, we have that $\displaystyle{\frac N{d^2} > \frac{N_0^3}{N_0^2} =N_0}$ and we can therefore apply inequality \eqref{eq:apply} and obtain that
\begin{equation} \label{eq:S1}
S_1 < \frac 1N \sum_{d=1}^{N_0}\frac{\varepsilon}4 \frac{N}{d^2} < \frac{\pi^2}6 \frac{\varepsilon}4 < \frac{\varepsilon}2.
\end{equation}
On the other hand, trivially,
\begin{equation} \label{eq:S2}
S_2 \le   \frac 1N \sum_{d=N_0+1}^N \frac N{d^2} < \int_{N_0}^N \frac{dt}{t^2} < \frac 1{N_0} < \frac{\varepsilon}2 .
\end{equation}
Gathering \eqref{eq:S1} and \eqref{eq:S2} in \eqref{eq:liou} proves estimate \eqref{eq:chowla-2} and therefore completes the proof of
Theorem \ref{thm:chowla}.
\end{proof}

\begin{rem} By a similar reasoning, one can show that Chowla's conjecture is in fact equivalent to the statement
$$\sum_{n\leq N}\lambda((n-1)!)\lambda((n+1)!) = o(N) \qquad (N\to\infty).$$
\end{rem}

\begin{rem}
It is clear that one can adapt the proof of Theorem \ref{thm:chowla} to prove that if $(a_n)_{n\ge 1}$ is any sequence of positive integers (not necessarily monotonic), then
\begin{equation} \label{eq:chowla-3}
\sum_{n\le N} \lambda(a_n)\lambda(n a_n) =o(N) \qquad (N\to \infty) .
\end{equation}
Hence, by choosing $a_n=(n-1)!$ in (\ref{eq:chowla-3}), we obtain (\ref{eq:chowla-1}).
On the other hand, by choosing $a_n=n$ and thereafter $a_n=P(n)$, we have
$$ \sum_{n\le N} \lambda(n) \lambda(n^2) =o(N) \qquad (N\to \infty) $$
and
$$ \sum_{n\le N} \lambda(P(n)) \lambda(nP(n)) =o(N) \qquad (N\to \infty). $$
Moreover, according to a more general version of Chowla's conjecture, given any positive integer $k$,
\begin{equation} \label{eq:chowla-general}
\sum_{n\le N} \lambda(n)\lambda(n+1)\cdots \lambda(n+k-1) = o(N) \qquad (N\to \infty).
\end{equation}
It is clear that one can adapt the proof of Theorem \ref{thm:chowla} to prove that if $k$ is an even integer, then the factorial version of (\ref{eq:chowla-general}) holds as well.
\end{rem}

\bibliographystyle{amsplain}

\begin{thebibliography}{99}
\bibitem{andrews} G.E. Andrews, B.C. Berndt, \emph{Ramanujan's lost notebook. {P}art {IV}}, Springer, New York, 2013.

\bibitem{alladi-erdos} K. Alladi,  P. Erd\H{o}s, \emph{On an additive arithmetic function}, Pacific J. Math. {\bf 71} (1977), 275-294.

\bibitem{BD} D. Berkane, P. Dusart, \emph{On a constant related to the prime counting function},
Mediterr. J. Math. {\bf 13} (2016), no. 3, 929–938.

\bibitem{chowla} S. Chowla, \emph{The Riemann hypothesis and Hilbert's tenth problem}, Mathematics and Its Applications, Vol. 4, Gordon and Breach Science Publishers, New York-London-Paris, 1965.

\bibitem{JMDK-1972} J.-M. De Koninck, \emph{On a class of arithmetical functions},
Duke Math. J., {\bf 39} (1972), 807-818.

\bibitem{JMDK-1974} J.-M. De Koninck, \emph{Sums of quotients of additive functions},
Proc. Amer. Math. Soc. {\bf 44} (1974), 35-38.

\bibitem{JMDK-2021} J.-M. De Koninck, \emph{A front row seat to the work of Aleksandar Ivi\'c}, Bulletin de l'Acad\'emie serbe des sciences et des arts {\bf 46} (2021), 13-43.

\bibitem{dkr} J.-M. De Koninck,  A.A.B. Razafindrasoanaivolala, \emph{On the middle divisors of an integer},  Ann. Univ. Sci. Budapest {\bf 50} (2020), 113--126.

\bibitem{jm-arthur-2023}
J.-M. De Koninck, A.A.B. Razafindrasoanaivolala, {\sl
On the quotient of the logarithms of the middle divisors of an integer}, Afr. Mat. {\bf 34}, no. 2 (2023), 1-15.

\bibitem{jm-a} J.-M. De Koninck and A. Ivi\'c, {\sl The distribution of the average
prime divisor of an integer}, Archiv der Math. {\bf 43} (1984), 37-43.


\bibitem{book-JM-AI} J.-M. De Koninck, A. Ivi\'c, \emph{Topics in Arithmetical
Functions}, North Holland, {\bf 43},  xvii + 262 pages, 1980.

\bibitem{DL} J.-M. De Koninck, F. Luca, {\it Analytic Number Theory: Exploring
the Anatomy of Integers}, Graduate Studies in Mathematics, Vol. 134, American Mathematical Society,
 2012.

\bibitem{1001} J.-M. De Koninck, A. Mercier, \emph{1001 Problems in Classical Number Theory}, American Mathematical Society,  Providence, RI, 2007, xii+336 pages.

\bibitem{pd} P. Dusart, \emph{Explicit estimates of some functions over primes}, Ramanujan J. {\bf 45} (2018), no. 1, 227-251.

\bibitem{dickson} L.E. Dickson, \emph{History of the Theory of Numbers}, Volume 1, Carnegie Institution of Washington, 1919.

\bibitem{egip} P. Erdös, S.W. Graham, A. Ivić, C. Pomerance, \emph{On the number of divisors of $n!$}, Analytic Number Theory: Proceedings of a Conference In Honor of Heini Halberstam Volume 1. Birkhäuser Boston, 337-355, 1996.

\bibitem{FE} G.V. Fedorov, \emph{Number of divisors of the central binomial coefficient}, Moscow Univ. Math. Bull. {\bf 68} (2013), 194–197.

\bibitem{ford} K. Ford, \emph{The distribution of integers with a divisor
in a given interval}, Ann. of Math. {\bf 168} (2008), 367-433.

\bibitem{HW} G.H. Hardy, E.M. Wright, \emph{An introduction to the theory of numbers}, Oxford University Press, 1979.

\bibitem{ivic-1980} A. Ivi\'c, \emph{Sums of reciprocals of the largest prime factor of an integer}, Arch. Math. {\bf 36} (1980), 57-61.

\bibitem{ivic-pomerance} A. Ivi\'c, C. Pomerance, \emph{Estimates of certain sums involving the largest prime factor of an
integer}, Coll. Math. Soc. J. Bolyai {\bf 34} in \emph{Topics in classical number theory}, Amsterdam, 1984.

\bibitem{ivic1} A. Ivi\'c, \emph{The {R}iemann zeta-function}, John Wiley \& Sons, New York, 1985.

\bibitem{ivic-2003} A. Ivi\'c, \emph{On certain large additive functions},  Paul Erd\H{o}s and his mathematics, I (Budapest, 1999), 319–331, Bolyai Soc. Math. Stud., 11, J\'anos Bolyai Math. Soc., Budapest, 2002.

\bibitem{Ipi} A. Ivi\'c, \emph{On a sum involving the prime counting function $\pi(x)$}, Univ. Beograd. Publ. Elektrotehn. Fak. Ser. Mat. {\bf 13} (2002), 85–88.

\bibitem{jakim3} R. Jakimczuk, \emph{Logarithm of the exponents in the prime factorization of the factorial}, Int. Math. Forum. Vol. 12. No. 13. 2017.

\bibitem{jakim} R. Jakimczuk,  \emph{Two Topics in Number Theory: Sum of Divisors of the Factorial and a Formula for Primes}, Int. Math. Forum, Vol. 12., No. 19., 2017.

\bibitem{jakim2} R. Jakimczuk,  \emph{Two Topics in Number Theory Sum of Divisors of the Factorial Part II
and Functions Defined on the
Divisors of a Number}, Preprint.

\bibitem{langevin} M. Langevin, \emph{Quelques applications de nouveaux r\'esultats de van der Poorten},
S\'eminaire Delange-Pisot-Poitou. Th\'eorie des nombres, tome 17, no. 2 (1975-1976),
exp. no G12, p. G1-G11.

\bibitem{nas} E. Naslund, \emph{The average largest prime factor}, Integers {\bf 13} (2013), 5 pages.

\bibitem{ram2} S. Ramanujan, \emph{Highly composite numbers}, Proc. London Math. Soc. (2) {\bf 14} (1915), 347-409.

\bibitem{ram1} S. Ramanujan, \emph{The lost notebook and other unpublished papers}, Narosa,
New Delhi, 1988.

\bibitem{tenenbaum} G. Tenenbaum, \emph{Sur deux fonctions de diviseurs}, J. London Math. Soc. {\bf 14} (1976), 521--526;
Corrigendum: J. London Math. Soc. {\bf 17} (1978), 212.

\bibitem{tij} R. Tijdeman, \emph{On the maximal distance between integers composed of small primes},
Compos. Math. {\bf 28} (1974), 159–162.

\bibitem{wig} S. Wigert,  \emph{Sur l’ordre grandeur du nombre de diviseurs d’un entier}, Ark. Math. 3, no. 18 (1907), 1-9.
\end{thebibliography}

\end{document}